\documentclass[11pt]{amsart}
\usepackage{amsmath,amssymb}
\usepackage{graphicx}
\usepackage{amsfonts,amscd,latexsym}

\newcommand{\labell}[1] {\label{#1}}

\newcommand{\1}{{{\mathchoice {\rm 1\mskip-4mu l} {\rm 1\mskip-4mu l}
{\rm 1\mskip-4.5mu l} {\rm 1\mskip-5mu l}}}}

\newcommand{\Cal}{{\rm Cal}}
\newcommand{\Sing}{{\rm Sing}}
\newcommand{\Ver}{{\rm Vert}}

\newcommand{\less}{{\smallsetminus}}

\newcommand{\bla}{{\bigl\langle}}
\newcommand{\bra}{{\bigl\rangle}}
\newcommand{\un}{\underline}
\newcommand{\Sp}{{\rm Sp}}
\newcommand{\Ta}{{\Tilde a}}
\newcommand{\oWw}{{\overline {\mathcal W}}}

\newcommand{\Tga}{{\Tilde\ga}}
\newcommand{\Tka}{{\Tilde\ka}}
\newcommand{\Tla}{{\Tilde\la}}
\newcommand{\Tr}{{\Tilde r}}
\newcommand{\Tc}{{\Tilde c}}
\newcommand{\TK}{{\Tilde K}}
\newcommand{\TQ}{{\Tilde Q}}

\newcommand{\p}{{\partial}}
\newcommand{\al}{{\alpha}}

\newcommand{\be}{{\beta}}

\newcommand{\Om}{{\Omega}}
\newcommand{\om}{{\omega}}
\newcommand{\eps}{{\varepsilon}}
\newcommand{\de}{{\delta}}
\newcommand{\De}{{\Delta}}
\newcommand{\ga}{{\gamma}}
\newcommand{\Ga}{{\Gamma}}
\newcommand{\io}{{\iota}}
\newcommand{\ka}{{\kappa}}
\newcommand{\la}{{\lambda}}
\newcommand{\La}{{\Lambda}}
\newcommand{\si}{{\sigma}}
\newcommand{\Si}{{\Sigma}}

\newcommand{\univ}{{\rm univ}}
\newcommand{\Uu}{{\mathcal U}}
\newcommand{\Aa}{{\mathcal A}}
\newcommand{\Ii}{{\mathcal I}}
\newcommand{\Ee}{{\mathcal E}}
\newcommand{\Gg}{{\mathcal G}}
\newcommand{\Ll}{{\mathcal L}}
\newcommand{\Bb}{{\mathcal B}}

\newcommand{\Oo}{{\mathcal O}}
\newcommand{\Mm}{{\mathcal M}}
\newcommand{\Rr}{{\mathcal R}}
\newcommand{\Ss}{{\mathcal S}}
\newcommand{\Xx}{{\mathcal X}}
\newcommand{\oMm}{{\overline {\Mm}}}
\newcommand{\ov}{\overline}

\renewcommand{\Tilde}{\widetilde}
\newcommand{\Ts}{{\Tilde s}}
\newcommand{\TM}{{\Tilde M}}
\newcommand{\TP}{{\Tilde P}}
\newcommand{\TF}{{\Tilde F}}

\newcommand{\Tom}{{\Tilde \om}}

\newcommand{\Tsi}{{\Tilde \si}}
\newcommand{\Tbe}{{\Tilde \be}}

\newcommand{\PP}{{\mathbb P}}

\newcommand{\NN}{{\mathbb N}}
\newcommand{\FF}{{\mathbb F}}
\newcommand{\Q}{{\mathbb Q}}
\newcommand{\R}{{\mathbb R}}
\newcommand{\C}{{\mathbb C}}

\newcommand{\Z}{{\mathbb Z}}

\newcommand{\Ham}{{\rm Ham}}

\newcommand{\Nn}{{\mathcal N}}
\newcommand{\Pp}{{\mathcal P}}
\newcommand{\Hh}{{\mathcal H}}

\newcommand{\Qq}{{\mathcal Q}}

\newcommand{\ev}{{\rm ev}}
\newcommand{\SSS}{{\smallskip}}
\newcommand{\QED}{{\hfill $\Box$\MS}}

\newtheorem{theorem}{Theorem}[section]
\newtheorem{thm}[theorem]{Theorem}

\newtheorem{cor}[theorem]{Corollary}
\newtheorem{lemma}[theorem]{Lemma}

\newtheorem{prop}[theorem]{Proposition}

\newtheorem{example}[theorem]{Example}

\newtheorem{rmk}[theorem]{Remark}

\numberwithin{figure}{section}
\numberwithin{equation}{section}
\numberwithin{table}{section}

\newcommand{\MS}{{\medskip}}

\newcommand{\NI}{{\noindent}}

\begin{document}
 \title{Hamiltonian $S^1$-manifolds are uniruled}
 \author{Dusa McDuff}\thanks{partially supported by NSF grant DMS 0604769}
\address{Department of Mathematics, Barnard College,
 Columbia University, 2990 Broadway, New York, NY 10027}
\email{dusa@math.columbia.edu}
\date{August 4, 2007, revised May 6 2008}
\begin{abstract}
The main result of this note is that every closed Hamiltonian 
$S^1$ manifold is uniruled, i.e. it has a nonzero Gromov--Witten invariant one of whose constraints is a point.
The proof uses the Seidel representation of $\pi_1$ of the Hamiltonian group in the small quantum homology of $M$ as well as the blow up technique recently introduced by Hu, Li and Ruan.  It applies 
more generally to manifolds that have  a loop of Hamiltonian
symplectomorphisms with a nondegenerate fixed maximum. Some consequences for Hofer geometry are explored. An appendix 
discusses the structure of the quantum homology ring of uniruled manifolds.
\end{abstract}

\maketitle

\tableofcontents

\section{Introduction}

A projective manifold is said to be {\bf projectively uniruled} if there is a holomorphic rational curve through every point. 
Projective manifolds with this property form an important class of  varieties in 
 birational geometry  since they do not have
minimal models but rather give rise to Fano fiber spaces;  see for example Kollar~\cite[Ch~IV]{K}.

One way 
to translate this property into the symplectic world
is to call a symplectic  manifold $(M,\om)$  {\bf (symplectically) uniruled} if there is a nonzero 
genus zero Gromov--Witten invariant of the form $\bla pt, a_2,\dots,a_k\bra_{k,\be}^{M}$ where $\be\ne 0$.
Here $k\ge 1$, $0\ne \be\in H_2(M)$, and $a_i\in H_*(M)$, 
and we consider the invariant where the $k$ 
marked points are allowed to vary freely.
Because 
Gromov--Witten invariants are preserved under symplectic deformation, a symplectically uniruled manifold  $(M,\om)$
has a $J$-holomorphic rational curve through every point for every $J$ that is tamed by some symplectic form deformation equivalent to $\om$.
Further justification for this definition is given in 
the foundational paper by Hu--Li--Ruan~\cite{HLR}. 
They show that $(M,\om)$ is symplectically uniruled whenever there is {\it any} nontrivial 
genus zero Gromov--Witten invariant $\bla \tau_{i_1}pt, \tau_{i_2} a_2,\dots,\tau_{i_k} a_k\bra_{k,\be}^{M}$ with $\be\ne 0$, where 
the ${i_j}\ge 0$ denote the degrees of the descendent insertions.\footnote
{
Descendent insertions  $\tau_i$  are defined in
 the discussion following equation (\ref{eq:GW}).}  
 They also show that 
the uniruled property is preserved 
under symplectic blowing up and down, and is
therefore  preserved by  symplectic birational equivalences.

We will call a symplectic manifold $(M,\om)$  {\bf strongly uniruled} 
if there is a nonzero invariant
$\bla pt, a_2, a_3\bra_{3,\be}^{M}.$  (Since one can always 
add marked points with insertions given by divisors, this is the same as requiring 
there be some nonzero invariant with $k\le 3$, but it is slightly different 
from Lu's usage of the term in~\cite[Def~1.14]{Lu}.) 
Manifolds with this
 property may be detected by the special structure of their small quantum cohomology rings; see Lemma~\ref{le:QH}.

 Kollar and Ruan showed in  \cite{K,Ru} (see also \cite[Thm.~4.2]{HLR}) that
every projectively uniruled  manifold is strongly uniruled; i.e. if there is a holomorphic $\PP^1$ through every point there is a nonzero 
invariant $\bla pt, a_2,a_3\bra_{\be}^{M}.$
 It is not 
clear that the same is true in the symplectic category.  Two questions are included here.  Firstly, there is a question about the behavior of Gromov--Witten invariants: if there is some nonzero $k$-point invariant with a point insertion must there be
 a similar nonzero $3$-point invariant?    Secondly, there is a more
  geometric question.  Suppose that $M$ is covered by $J$-holomorphic $2$-spheres either 
for one $\om$-tame $J$ or  for a 
significant class of $J$.  Must $M$ then be uniruled? Hamiltonian $S^1$-manifolds are a good test case here since, when $J$ is $\om$-compatible and $S^1$-invariant, 
 there is an $S^1$-invariant $J$ sphere through every point (given by the orbit of a gradient flow trajectory of the moment map with respect to the associated metric $g_J$.)

In this note we show that every 
 Hamiltonian $S^1$-manifold is  uniruled.  This is obvious when 
 $n=1$ and is well known for $n=2$ since, by Audin \cite{Au} and Karshon~\cite{Ka}, the only $4$-dimensional Hamiltonian $S^1$-manifolds are blow ups of rational or ruled surfaces.
Our proof in higher dimensions  relies heavily on the approach
used by Hu--Li--Ruan~\cite{HLR} to analyse the 
Gromov--Witten invariants of a blow up.    

 The first step is to argue  as follows.  By Lemma~\ref{le:semi} we 
 may
  blow  $M$ up along its maximal and minimal fixed point sets 
  $F_{\max}, F_{\min}$ until these two are divisors.  Then consider the gradient flow of the (normalized) moment map $K$ with respect to
  an $S^1$-invariant metric $g_J$ constructed from an $\om$-compatible invariant almost complex structure $J$.   The $S^1$-orbit of any gradient flow line is $J$-holomorphic.  If $\al$ is the
 class  of the $S^1$-orbit of a  flow line  from $F_{\max}$ to $F_{\min}$ then $c_1(\al) = 2$, and because there is just one of these spheres through a generic point of $M$
  one might 
  naively think that the Gromov--Witten invariant 
  $\bla pt, F_{\min},F_{\max}\bra^M_\al$ is $1$.  However, there could be other  curves in class $\al$ that cancel this one.  Almost the only case in which one can be sure this does not happen is when the $S^1$ action is semifree, i.e. no point in $M$ has finite stabilizer: see Proposition~\ref{prop:al}.
 If the order of the stablizers is at most $2$, then 
 one can also show that $(M,\om)$ must be
   strongly uniruled:
  see Proposition~\ref{prop:2iso}.  However it is not clear whether $(M,\om)$ must be strongly uniruled when the isotropy has higher order.
  
To deal with the general case, we 
use the 
Seidel representation 
of $\pi_1(\Ham(M,\om))$ in the group of multiplicative units 
of the quantum homology ring $QH_*(M)$.\footnote
{
Although we describe this here in terms of genus zero Gromov--Witten invariants,  another way to think of it is as the transformation  induced on Hamiltonian Floer homology by continuation along noncontractible Hamiltonian loops.}
  To make the argument work,
we blow up once more at a point in $F_{\max}$, obtaining an 
$S^1$-manifold called $(\TM,\Tom)$.  
By Proposition~\ref{prop:SU} the Seidel element $\Ss(\Tga)\in QH_{2n}(\TM)$ of the resulting $S^1$ action $\Tga$ on $\TM$ 
involves the exceptional divisor $E$ on $\TM$.  Although we know very little about the structure of $QH_*(M)$, the fact $(M,\om)$ is not uniruled implies by Proposition~\ref{prop:unicond2} that the part of $QH_*(\TM)$ that does not involve $E$ forms an ideal.  Moreover, the quotient of
$QH_*(\TM)$  by this ideal has an understandable structure.  Using this, we show that  the invertibility of  $\Ss(\Tga)$ implies that certain terms in
the inverse element
 $\Ss(\Tga^{-1})$ cannot vanish.  
 This tells us that certain section invariants\footnote
 {
 Suppose that $P\to S^2$ is a bundle with fiber $(M,\om)$ 
 with Hamiltonian structure group.  Then the fiberwise symplectic form $\om$ extends to a symplectic form $\Om$ on $P$.  Gromov--Witten invariants of $P$ in classes $\be\in H_2(P;\Z)$ that project to the positive generator of $H_2(S^2;\Z)$ are called {\it section invariants}, while those whose class lies in the image of $H_2(M;\Z)$ 
 are called {\it fiber invariants}.}
 of the fibration $\TP'\to S^2$  defined by the loop $\Tga^{-1}$ cannot vanish.
 The homological constraints here involve $E$. The final step is to show
 that these invariants can be nonzero only if $(M,\om)$ is uniruled.
 Hence the original manifold is as well, by the blow down result of  Hu--Li--Ruan~\cite[Thm~1.1]{HLR}.  Their blowing down argument does not give control on the number of insertions; a $3$-point invariant might blow down to an invariant with more insertions. Hence we cannot conclude that   $(M,\om)$ is strongly uniruled.

This argument does not use the properties of $F_{\min}$ nor does it use 
much about the circle action.  We do need to assume that the loop
$\ga = \{\phi_t\}$ has a {\bf fixed maximal submanifold}; i.e. that there is
a  nonempty (but possibly disconnected)  submanifold $F_{\max}$ such that at each time $t$
the generating Hamiltonian $K_t$ for $\ga$  takes its maximum 
on $F_{\max}$ 
in the strict sense that if $x_{\max}\in F_{\max}$ then $K_t(x)\le K_t(x_{\max})$ for all $x\in M$
 with equality iff $x\in F_{\max}$.  Further we need $\ga$ to restrict to an $S^1$ action near $F_{\max}$; i.e.  in appropriate coordinates $(z_1,\dots z_k)$ normal to $F_{\max}$ we assume that near $F_{\max}$ 
 the generating Hamiltonian $K_t$ has the form
 $const - \sum m_i |z_i^2|$  for some positive integers $m_i$.
  In these circumstances, we shall say that $\ga$ is {\bf a circle action near its maximum.} Note that this action is effective (i.e. no element other than the identity acts trivially) iff gcd$(m_1,\dots,m_k) = 1$.
 
 Our main result is the following.
 
\begin{thm}\labell{thm:main}  Suppose that $\Ham(M)$ contains a loop $\ga$ with a fixed maximum near which $\ga$ 
is an effective  circle action.
Then $(M,\om)$ is uniruled. \end{thm}

\begin{cor}\labell{cor:main}  Every Hamiltonian $S^1$-manifold is uniruled.
\end{cor}

\begin{rmk}\rm   Since our proofs involve the blow up of $M$ they work only when $n: = \frac 12 \dim M \ge 2$. However 
Theorem~\ref{thm:main} 
is elementary when  $n=1$.  For then, if $M\ne S^2$, the group 
$\Ham M$ is contractible, as is each component of 
the group $G$ of elements in $\Ham M$ that fix $x_0$. 
Therefore the homomorphism $\pi_1(G)\to\pi_1(\Sp(2,\R))$ given by 
taking the derivative at $x_0$ is trivial.   This implies that there is no loop $\ga$ in $\Ham M$ that is 
a nonconstant circle action near $x_0$.  

In fact, in this case there is also no nonconstant loop that is the identity
 near a fixed maximum.  For $
\om$ is exact on  $M\less \{x_0\}$ so that 
the Calabi homomorphism $\Cal: \Ham^c(M\less \{x_0\}) \to \R$ is  well defined  on the group   
$\Ham^c$  of compactly supported 
Hamiltonian symplectomorphisms of $M\less \{x_0\}$; cf. \cite[Ch.~10.3]{MS1}.  In particular
$$
\Cal(\phi): = \int_0^1\Bigl(\int_M\bigl(H_t(x) - H_t(x_0)\bigr)\,\om\Bigr)dt,
$$
 is independent of the choice of path $\phi^H_t$ 
in $\Ham^c$ with time $1$ map $\phi$.  But if $\ga$ were a
nonconstant loop in $\Ham^c$ with fixed maximum at $x_0$ we could arrive at $\phi=id$ by a path for which this integral is negative. \end{rmk}


If all we know is that $\ga$ has a fixed maximum $F_{\max}$ then
for each $x\in F_{\max}$ the linearized flow $A_t(x)$ in the 
$2k$-dimensional normal
space $N_x^F: =T_xM/T_x(F_{\max})$ is 
generated by a family of nonpositive  quadratic forms. We shall say that $F_{\max}$ is {\bf nondegenerate} if these quadratic forms are everywhere negative definite. In this case, the linearized flow is 
a so-called positive loop in the symplectic group $\Sp(2k;\R)$, i.e. a loop generated by a family of negative definite quadratic forms;
cf. Lalonde--McDuff~\cite{LM}.    If $2k\le 4$ Slimowitz~\cite[Thm~4.1]{Sl} shows that any two positive loops that are homotopic in $\Sp(2k;\R)$ are in fact homotopic through a family of positive loops. 
(In principle this should hold in all dimensions, but the details have been worked out only in low dimensions.)  
 By 
Lemma~\ref{le:pos} below, 
in the $4$-dimensional case  one can then homotop $\ga$ so that it 
is an effective  local circle action near its maximum.  Thus we find:

\begin{prop}\labell{prop:loop} 
 Suppose that $\dim M \le 4$ and the loop $\ga$ has a
nondegenerate maximum at $x_{\max}$.  Then $\ga$ can be homotoped so that it is 
a nonconstant circle action near its maximum, 
that when $\dim M = 4$ can be assumed  effective.
Thus $(M,\om)$ is uniruled.
\end{prop}

Slimowitz's result also implies that  
 every $S^1$ action
on $\C^k$  with strictly positive weights $(m_1,\dots,m_k)$ is homotopic through positive loops  to an action with positive  weights
$(m_1',m_2',m_3\dots,m_k)$  where $m_1',m_2'$ are mutually prime.
(Here we do not change the action on the last $(k-2)$ coordinates.)  Since actions are effective iff  their  weights at any fixed point 
are mutually prime we deduce:

\begin{cor}  Theorem~\ref{thm:main} holds for any loop that is a circle action near its maximum.
\end{cor}

These results have implications for Hofer geometry. 
It is so far unknown which elements in $\pi_1(\Ham(M))$
can be represented by loops that minimize the Hofer length.
 By Lalonde--McDuff~\cite[Prop.~2.1]{LMe}  
such a loop 
has to have a fixed maximum and minimum
$x_{\max}, x_{\min}$, i.e. points such that the generating Hamiltonian 
$K_t$ satisfies the inequalities $K_t(x_{\min})\le K_t(x)\le K_t(x_{\max})$ for all $x\in M$ and $t\in [0,1]$. 
However, these extrema could be degenerate and need not form the same
manifold for all $t$.

If fixed extrema exist, then the
 $S^1$-orbit of a path from $x_{\max}$ to $x_{\min}$ forms a $2$-sphere, and it is easy to check that the integral of $\om$ over this $2$-sphere is 
just $\|K\|: = \int \bigl(K_t(x_{\max})-K_t(x_{\min})\bigr) dt$; cf.\cite[Ex.9.1.11]{MS}.
Hence such a loop cannot exist on a symplectically 
aspherical manifold $(M,\om)$.  However,  so far no examples
are known  of symplectically aspherical manifolds with nonzero
$\pi_1(\Ham(M))$.

It is easy to construct nonzero elements in $\pi_1(\Ham(M))$
for manifolds that are not uniruled.  For example, by \cite[Prop~1.10]{Mcbl} one could take $M$ to be the two 
point blow up of any symplectic $4$-manifold.  Hence, if one could better understand the situation when the fixed points are degenerate, it might be possible to exhibit manifolds for which $\pi_1(\Ham)$ is nontrivial
 but where no nonzero element in this group has a length minimizing representative.  This question is the subject of ongoing research. 
 
 \begin{rmk}\labell{rmk:max}\rm (i) 
 It is not clear whether a
Hamiltonian $S^1$-manifold $(M,\om)$ always has a blow up that is strongly uniruled. However,
 as we show in Proposition~\ref{prop:a}  there is not much 
 difference between the uniruled and the strongly uniruled conditions.
 Moreover  by  Corollary~\ref{cor:coh} they coincide
  in the case when  $H^*(M;\Q)$ is generated by $H^2(M;\Q)$.
  \SSS
 
 \NI (ii)  It is essential in Corollary~\ref{cor:main} 
 that the loop is  Hamiltonian; it is not enough that it has fixed points.  For example,  the semifree symplectic circle action
constructed in~\cite{Mcs} has fixed points but these 
 all have index and coindex equal to 
$2$.  Hence the first Chern class $c_1$ vanishes on $\pi_2(M)$ and
 $M$ cannot be uniruled for dimensional reasons.\SSS
 
 \NI (iii)  One might wonder why  we work with the blow up $\TM$ rather than with $M$ itself.   One answer is seen in results such as
Lemma~\ref{le:U} and Propositions~\ref{prop:SU}
and ~\ref{prop:eps}.  These
 make clear that the 
 exceptional divisor in $\TM$ forms a visible marker that allows us to show that certain Gromov--Witten invariants of $M$ do not vanish. 
The point is that
  even though we cannot calculate $QH_*(\TM)$ in general, the results of
  \S\ref{ss:qh} show that
  we can calculate the quotient $QH_*(\TM)/\Ii$ provided that $(M,\om)$ is {\it not} uniruled.
  
 As always, the problem is  that it is very hard to calculate 
 Gromov--Witten
 invariants for a general symplectic manifold; if one has no global information about the manifold, the contribution  to an invariant of a known $J$-holomorphic curve  could very well be cancelled by some other unseen curve.
 The Seidel representation is one of the very few general tools that 
  can be used to show that certain invariants of an arbitrary symplectic
   manifold cannot vanish and hence it has found several interesting applications; cf. Seidel~\cite{Sei2} and Lalonde--McDuff~\cite{LM2}.
 It is made by counting section invariants of Hamiltonian bundles $P$ with fiber $M$ and base $S^2$. It is fairly clear that the existence of  fiber invariants of the blow up bundle $\TP$ that involve the exceptional divisor $E$ should imply that $(M,\om)$ is uniruled.  A crucial step in our argument is Lemma~\ref{le:horpt} 
 which states  that the existence of certain section invariants
of $P$ also implies that $(M,\om)$ is uniruled.
 
Another way of getting global information on Gromov--Witten invariants is
 to use  symplectic field theory.  
This approach was recently taken by J. He~\cite{He} who shows
that subcritical manifolds are strongly uniruled, thus proving
 a conjecture of Biran and Cieliebak in
   \cite[\S9]{BC}.   This case is somewhat different than the one considered here; for example $\PP^n$ is subcritical while $\PP^1\times \PP^1$ is not.
  \end{rmk}

The main proofs are given in \S\ref{s:S1}. They use some properties  
 of absolute and relative 
Gromov--Witten invariants established in 
\S\ref{s:rel}.  
Conditions that imply $(M,\om)$ is strongly uniruled are 
discussed in \S\ref{s:fur}.  The appendix treats general questions about the structure of the quantum homology ring.

\begin{rmk}[Technical underpinnings of the proof]\labell{rmk:tech}\rm The main technical tool used in this paper
 is the decomposition rule for relative genus zero 
Gromov--Witten invariants. Since we use this with descendent absolute insertions, it is useful to have  this in a strong form in which the virtual moduli cycles in question are at least $C^1$.  To date the most relevant references for this result are Li--Ruan~\cite[Thm~5.7]{LR} and Hu--Li--Ruan~\cite[\S3]{HLR}.  However, relative Gromov--Witten theory is a part of the symplectic field theory package (cf. \cite{EGH,BE}).
Therefore the work of Hofer--Wysocki--Zehnder~\cite{H,HWZ1,HWZ2,HWZ3} on polyfolds  should eventually lead to an alternative proof.  
In \S\ref{s:rel} we explain the relative moduli spaces, but our proofs assume the basic results about them.  Note that we cannot get anywhere by restricting to the semi-positive case (a favorite way of avoiding virtual cycles)
since we need to consider almost complex structures on $M$ that are $S^1$-invariant and these are almost never regular unless the action is semifree.

Another important 
ingredient of the proof is the Seidel representation of $\pi_1$ of the Hamiltonian group $\Ham(M,\om)$, which is proved for general symplectic manifolds in
~\cite[\S1]{Mcq}. Finally, we use the identities (\ref{eq:LPi}) and (\ref{eq:LPii}) for genus zero Gromov--Witten invariants.  
The former is fairly well known and is not hard to deduce from the corresponding result for genus zero stable curves, while the latter
was first proved by Lee--Pandharipande~\cite[Thm~1]{LP} in the algebraic case.  For Theorem~\ref{thm:main} we only need a very special 
case of this identity
whose proof we explain in Lemma~\ref{le:LP}.  The full force of
(\ref{eq:LPii}) is used to prove Proposition~\ref{prop:usu}, 
but this result is not central to the main argument.
\end{rmk}

\NI {\bf Acknowledgements}   Many thanks to Tian-Jun Li, 
Yongbin Ruan, Rahul Pandharipande  and Aleksey Zinger for useful discussions
about relative Gromov--Witten invariants.  I also thank 
the first two named above as well as Jianxun Hu for showing me 
early versions of their paper~\cite{HLR}, and Mike Chance and 
the referees for various
helpful comments.  Finally I wish to thank MSRI for their May 2005 conference on Gromov--Witten invariants and Barnard College and Columbia University for their hospitality during the final stages
 of my work on this project.

\section{The main argument}\labell{s:S1}

Because the main argument is somewhat involved, we shall begin by outlining it. In \S\ref{ss:qh}
we explain the structure of the quantum homology ring $QH_*(\TM)$
of the one point blow up of a manifold $(M,\om)$ that is not  uniruled.
By Proposition~\ref{prop:unicond2} this condition on $(M,\om)$ implies that the subspace of $QH_*(\TM)$ spanned by the
homology classes in $H_*(M\less pt)$ forms an ideal $\Ii$.  Moreover,
the quotient $\Rr: = QH_*(\TM)/\Ii$ decomposes as the sum of two fields (Lemma~\ref{le:R}), and we can work out an explicit formula for the inverse $u^{-1}$ of any unit $u$ in $\Rr$.

The unit in question is the Seidel element $\Ss(\Tga)\in QH_*(\TM)$ of the Hamiltonian loop $\Tga$ on $\TM$.  Usually it is very hard to calculate $\Ss(\Tga)$.  However, when the maximum fixed point set $\TF_{\max}$ of $\Tga$ is a divisor, one can calculate 
its leading term (Corollary~\ref{cor:S1}).  Moreover, because the final blow up is at a point of $F_{\max}$ we show in Proposition~\ref{prop:SU} that this leading term has 
nontrivial image in $\Rr$. 
We next use  the explicit formula for $u^{-1}\in \Rr$ to deduce
the nonvanishing of certain terms in the Seidel element 
$\Ss(\Tga')$ for the inverse loop $\Tga': = \Tga^{-1}$.
But $\Ss(\Tga')$ is calculated from the
Gromov--Witten invariants
of the fibration $\TP'\to \PP^1$ determined by $\Tga'$.  Hence we  deduce in Corollary~\ref{cor:SU} that some of these invariants do not vanish. 

Finally we show that these 
particular invariants can be nonzero only if $M$ is uniruled, so that
our initial assumption that $(M,\om)$ is not uniruled is untenable.
The proof here involves two steps.  First, by blowing down the bundle 
$\TP'$ to $P'$, the bundle determined by $\ga': = \ga^{-1}$, we deduce the nonvanishing of certain  
Gromov--Witten invariants in $P'$ (Proposition~\ref{prop:eps}).  
Second, we show in Lemma~\ref{le:horpt} that these invariants 
can be nonzero only if $(M,\om)$ is uniruled.

In this section we prove essentially all results that do not involve 
comparing the Gromov--Witten invariants of a manifold 
and its blow up. (Those proofs are deferred to \S\ref{s:rel}.)
We assume throughout that $(M,\om)$ 
is a closed 
symplectic manifold of dimension $2n$ where $n\ge 2$.

\subsection{Uniruled manifolds and their pointwise blowups}\labell{ss:qh}

Consider the small quantum homology  $QM_*(M):= H_*(M)\otimes \La$ of $M$.  
 Here we use the Novikov ring  $\La: = \La_\om [q,q^{-1}]$ where $q$ is a polynomial
 variable of degree $2$
 and  $\La_\om$ denotes the
generalized Laurent series ring with elements 
$\sum_{i\ge 1} r_it^{-\ka_i}$, 
where $r_i\in \Q$ and $\ka_i$ is a strictly
 increasing sequence that 
tends to $\infty$ and lies in the period subgroup $\Pp_\om$ of $[\om]$, i.e. the image of the homomorphism $I_\om: \pi_2(M)\to \R$ given by integrating $\om$.
We assume that $[\om]$ is chosen so that $I_\om$ is injective.
 We write the elements of
 $QH_*(M)$ as infinite sums $\sum_{i\ge 1} a_i\otimes q^{d_i}t^{-\ka_i}$, where $a_i\in H_*(M;\Q)=:H_*(M)$, the $|d_i|$ are bounded  and $\ka_i$ is as before.
 The term $a\otimes q^{d}t^{\ka}$ has degree $2d+\deg a$.
 
 The quantum product $a*b$ of the elements $a,b\in 
 H_*(M)\subset QH_*(M)$ is defined as follows.  Let $\xi_i, i\in I$, be a basis for $H_*(M)$ and write $\xi_i^*, i\in I,$ for the basis of $H_*(M)$ that is dual with respect to the 
 intersection pairing,  that is $\xi_j^*\cdot\xi_i = \de_{ij}$.  Then 
\begin{equation}\labell{eq:QH}
 a*b: = \sum_{i,\be \in H_2(M;\Z)} \bla a,b,\xi_i\bra^M_\be\,\, \xi_i^*\,\otimes q^{-c_1(\be)} t^{-\om(\be)},
\end{equation}
 where $\bla a,b,\xi_i\bra^M_\be$ denotes the Gromov--Witten invariant in $M$ that counts curves in class $\be$ through the homological
 constraints
 $a,b,\xi_i$.
 Note that if $(a*b)_\be: = \sum_i \bla a,b,\xi_i\bra^M_\be \xi_i^*$, then $(a*b)_\be\cdot c = \bla a,b,c\bra^M_\be$. Further, 
 $\deg(a*b) = \deg a + \deg b - 2n$, and the identity element is $\1: = [M]$.

   \begin{lemma}\labell{le:QH} The following conditions are equivalent:
\MS

\NI
{\rm (i)}  $(M,\om)$ is not strongly uniruled;\SSS

\NI
{\rm (ii)}  $\Qq_-: = \underset{i<2n}\oplus  H_i(M)\otimes \La$ 
is an ideal in the small quantum homology $QH_*(M)$;\SSS

\NI
{\rm (iii)}  $pt*a=0$ for all $a\in \Qq_-$,  where $*$ is the quantum product.\SSS

\NI
Moreover, if these conditions hold every unit $u\in QH_*(M)^{\times}$ in $ QH_*(M)$ has the form
$u=\1\otimes\la + x$, where $ x\in \Qq_-$ and $\la$ is a unit in  $\La$.
\end{lemma}

\begin{proof}    Choose the basis $\xi_i$ so that $\xi_0=pt$, $\xi_M=[M]=:\1$ and all other $\xi_i$ lie in $H_j(M)$ for $1\le j<2n$.
Then $\xi_0^* = \xi_M$, and
the coefficient of $\1\otimes q^{-d} t^{-\ka}$ in $a*b$ is
 $$
\sum_{\be:\om(\be)=\ka,c_1(\be)=d}
 \bla a,b,pt\bra_\be. 
$$ 
Since we assume that $I_\om:\pi_2(M)\to \Ga_\om$ is an isomorphism
and $\ka\in \Ga_{\om}$, there is precisely one class $\be$ such that
$\om(\be)=\ka$.  
If $\be = 0$ this invariant is the usual intersection product
$a\cap b\cap pt$ and so always vanishes for $a,b\in H_{<2n}(M)$.  
Therefore,  this coefficient is nonzero for some such $a,b$  iff
 $(M,\om)$ is strongly uniruled. But this coefficient is nonzero 
 iff $\Qq_-$ is not a subring.  Note finally that since  $\Qq_-$ has codimension $1$, it is an ideal iff it is a subring.  
 
 This proves the equivalence of (i) and (ii).
 The other statements are proved by similar arguments. 
\end{proof}

\begin{rmk}\label{rmk:app}\rm   It seems very likely that $(M,\om)$ is strongly uniruled iff
$\Qq_-$ contains no units.   We prove this in the appendix, 
as well as a generalization to uniruled manifolds, in the case when
the odd Betti numbers of $M$ vanish.
\end{rmk}

Now let $(\TM,\Tom)$ be the $1$-point blow up of $M$ with 
exceptional divisor  $E$.  Put $\eps: = E^{n-1}$, the class of a line in $E$. The blow up parameter is $\de: = \Tom(\eps)$, 
which we assume to be $\Q$-linearly independent from $\Ga_{\om}$.
Later it will be convenient to write
$\La_{\Tom}=:\La_{{\om,\de}}$.

\begin{lemma}\labell{le:GWE} 
 Invariants of the form
$$
  \bla E^i,E^j,E^k\bra^{\TM}_{p\eps},\qquad 0< i,j,k<n,\; p>0  $$
are nonzero  only if $p=1$ and $i+j+k=2n-1$, and in this case the invariant is $-1$.
  \end{lemma}
  \begin{proof}   Since $c_1(\eps) = n-1$  the invariant
  $  \bla E^i,E^j,E^k\bra^{\TM}_{p\eps}$  is nonzero only if
  $n + p(n-1) = i+j+k$. Since $i+j+k\le 3n-3$ we must have $p=1$.
  The result now follows from the fact that
if $J$ is standard 
near $E$ then the only $\eps$ curves are lines in $E$.
\end{proof}

\begin{cor}\labell{cor:GWE} Let $0\le i,j < n$.  Then
$$\begin{array}{lcl}
E^i* E^j = E^i\cap E^j = E^{i+j}&\mbox{if} &i+j<n,\\
E^i*E^{n-i} = - pt + E\otimes q^{-n+1}t^{-\de}, &&\mbox{and}\\
E^i*E^j= E^{i+j-n+1}\otimes  q^{-n+1}t^{-\de}&\mbox{if} &n<i+j<2n-1.\end{array}
$$
\end{cor}

We prove the next result  in \S\ref{s:rel} by the technique of Hu--Li--Ruan~\cite{HLR}.

 \begin{prop}\labell{prop:unicond2} Let $(\TM,\Tom)$ be the one point blow up of $M$ with 
exceptional divisor  $E$, and let $a,b\in H_{<2n}(M)$.
 If any  invariant of the form
 $$
  \bla a,b,E^k\bra^{\TM}_\Tbe,\quad  \bla a,E^i,E^j\bra^{\TM}_\Tbe,\quad \bla E^i,E^j,E^k\bra^{\TM}_\Tbe,\quad \mbox{for }\; 0\ne\Tbe\ne p\eps,\;\;\; i,j,k\ge 1
  $$
 is nonzero, then  $(M,\om)$ is  uniruled.
\end{prop}

 Decompose $QH_*(\TM)$ additively as $\Ii\oplus \Ee$ where
$\Ii$ is  generated as a $\La_{\Tom}$-module by the image of $H_{<2n}(M)$ in $H_*(\TM)$ and $\Ee$ is spanned by $\1,E,\dots, E^{n-1}=\eps$. Proposition~\ref{prop:unicond2} implies:

\begin{cor}  If $(M,\om)$ is not uniruled,
$\Ii$ is an ideal in $QH_*(\TM)$.
\end{cor}
\begin{proof}  Let $a,b\in H_{<2n}(M)$.  Then $a*b\in \Ii$ iff
$(a*b)_\Tbe\cdot E^k = \bla a,b,E^k\bra^{\TM}_\Tbe = 0$ for all $\Tbe\in H_2(\TM)$ and all $k\ge 1$.  Therefore the hypotheses imply that $\Ii$ is a subring.  It is an ideal because $a*E^k=0$ for all $k\ge 1$.
\end{proof}

Let $\Ee_{2n}$ denote the degree $2n$ part of $\Ee$. 
This is not a subring of $QH_*(\TM)$ since for example
$$
(E\otimes q) *(\eps\otimes q^{n-1}) = - pt\otimes q^{n} + E\otimes qt^{-\de}.
$$  
Nevertheless if $(M,\om)$ is not uniruled, 
$\Ee_{2n}$ can be given the ring structure of the quotient
$QH_{2n}(\TM)/\Ii$.  By Corollary~\ref{cor:GWE} and Proposition~\ref{prop:unicond2}, there is a ring isomorphism
$$
\Ee_{2n}\to  \La_{{\om,\de}}[s]/\bigl(s^n= st^{-\de}\bigr)
\quad \mbox{given by }\;\;E\otimes q\;\mapsto\; s.
$$
Composing with the quotient map $QH_{2n}(\TM)\to \Ee_{2n}\cong QH_{2n}(\TM)/\Ii$ we get a homomorphism 
$$
\Phi_E:
QH_{2n}(\TM)\to \Rr: = \La_{\om,\de}[s]/\bigl(s^n= st^{-\de}\bigr).
$$

\begin{lemma}\labell{le:R}  The ring $\Rr$ decomposes 
as the direct sum of two fields $\Rr_1\oplus \Rr_2$.
\end{lemma} 
\begin{proof} Denote $e_1: = 1-s^{n-1}t^\de$ and $e_2: = s^{n-1}t^\de$.
Because $se_1=0, se_2=s$ we find
 $e_i^2=e_i$ and $e_1e_2=0$.  Moreover
$\Rr_1:= e_1\Rr$ is isomorphic to the field
$e_1\La_{{\om,\de}}$.   To see that $\Rr_2: = e_2\Rr$ is a field,
consider the homomorphism
$$
F:\Rr\to \La_{\om,{\de}/(n-1)}\quad \mbox{given by }\;\;
s\mapsto t^{-{\de}/(n-1)},t\mapsto t.
$$
Then $F(e_2) = 1$ and the kernel of $F$ is $e_1\Rr$.  Hence $F$ gives an isomorphism between $e_2\Rr$ and the field $\La_{\om,{\de}/(n-1)}$.
\end{proof}

Denote 
$$
\Xx:= \bigl\{x\in \Rr:x=\sum_{i\ge0} r_is^{d_i}t^{-\ka_i},\;r_i\in\Q, 
 \,0\le d_i\le n,\,\ka_i>0\bigr\}.
$$
Then for all  $x\in \Xx$ the element
 $1+x$ is invertible with inverse $1-x+x^2-\dots$.

\begin{lemma}\labell{le:U}  Let $u\in\Rr$ be a unit of the form 
$
1+ rst^{\ka_0}(1+x)$ where $x\in \Xx$, $r\ne 0$ and 
$\ka_0\in \La_\om$ is  $>\de$.  
\SSS

\NI (i) 
If $n\ge 3$ then
$$
u^{-1} = 1 - s^{n-1}t^\de + \frac 1r s^{n-2}t^{\de-\ka_0}\bigl(1+y\bigr),\quad \mbox{for some }\,y\in \Xx.
$$
\NI (ii)  If $n=2$ then $ u^{-1} = 1 - st^\de + \frac 1r st^{2\de-\ka_0}
\bigl(1+y\bigr)$ for some $y\in \Xx$.
\end{lemma} 
\begin{proof} Since $e_2u = e_2(e_2 + rst^{\ka_0}(1+x))$ we may 
write 
\begin{eqnarray*}
u \;=\; e_1u + e_2u &=& e_1 + e_2\Bigl(s^{n-1}t^\de + rst^{\ka_0}(1+x)\Bigr)\\
&=& e_1 + e_2rst^{\ka_0}(1 + x') \quad \mbox{where }\,x':=x+\frac 1r s^{n-2} t^{\de-\ka_0}.
\end{eqnarray*}
If $n>2$ define
 $v:  = e_1 + \frac 1r e_2 s^{n-2}t^{\de-\ka_0}(1+x')^{-1}$.
Then $e_1uv = e_1$ while $e_2uv = e_2s^{n-1}t^\de = e_2.$ 
Hence $u^{-1}=v$.  
Since $e_2 s^{n-2}= s^{n-2}$ when $n>2$, this has the form required  in (i).
When $n=2$ we take $v=   e_1 + \frac 1r e_2 st^{2\de-\ka_0}(1+x')^{-1}$.
\end{proof}

\begin{rmk}\labell{rmk:univ}\rm   Here we restricted the coefficients 
of $QH_*(M)$ to $\La_\om[q,q^{-1}]$ to make the structure of $\Rr$ as simple as possible.  However, in order to define the Seidel representation we will need to use the larger coefficient ring  $\La^{\univ} [q,q^{-1}]$ where $\La^{\univ}$ consists of all formal  series
$\sum_{i\ge 1} r_it^{-\ka_i}
$, 
where  $\ka_i\in \R$ is any increasing sequence that 
tends to $\infty$.  Since $\Rr$ injects into $\Rr': = 
\La^{\univ}[s]/\bigl(s^n= st^{-\de}\bigr)$ the statement in 
Lemma~\ref{le:U} remains valid when we think of $u$  as an element of $\Rr'$.  \end{rmk}

\subsection{The Seidel representation}\labell{ss:S}

This is a homomorphism $\Ss$ from $\pi_1(\Ham(M,\om))$ to the degree $2n$ multiplicative units $QH_{2n}(M)^\times$ 
of the small quantum homology ring first considered by 
Seidel in~\cite{Sei}.  To define it, observe that  each loop $\ga=\{\phi_t\}$ in  $\Ham(M)$ gives rise to an 
$M$-bundle
$P_\ga\to \PP^1$ defined by the clutching function $\ga$:
$$
P_\ga: = M\times D_+\cup M\times D_-/\!\sim \quad\mbox{where} \;\;\big(\phi_t(x),e^{2\pi it}\big)_+\sim
\big(x,e^{2\pi it}\big)_-. 
$$
Because the loop $\ga$ is Hamiltonian, the fiberwise symplectic form $\om$ extends to a closed form $\Om$ on $P$, that we can arrange to be symplectic by adding to it the pullback of a suitable form on the base $\PP^1$; see the proof of Proposition~\ref{prop:max} below.

In the case of a circle action with normalized moment map $K:M\to \R$ we may simply take
$(P_\ga, \Om)$ to be the quotient $(M\times_{S^1} S^3, \Om_c)$, where $S^1$ acts diagonally on $S^3$ and $\Om_c$ pulls back to
$\om + d\bigl((c-K)\al\bigr)$.  Here $\al$ is the standard contact form on $S^3$ normalized so that it descends to an area form on $S^2$ with total area $1$, and $c$ is any constant larger than the maximum $K_{\max}$ of $K$. Points $x_{\max}, x_{\min}$  in the fixed point sets $F_{\max}$ and $F_{\min}$ give rise to sections
$s_{\max}: =x_{\max}\times \PP^1$ and $s_{\min}:= x_{\min}\times \PP^1$.  Note that 
our orientation conventions are chosen so that the integral 
of $\Om$ over the section $s_{\min}$ is {\it larger} than that over $s_{\max}$. For example, if $M=S^2$ and $\ga$ is a full rotation,
$P_\ga$ can be identified with the one point blow up of $\PP^2$, and $s_{\max}$ is the exceptional divisor. In the following we denote particular sections as $s_{\max}$ or $s_{\min}$, while writing $\si_{\max},\si_{\min}$ for the homology classes they represent.

The bundle $P_\ga\to \PP^1$ carries two canonical cohomology classes, the first Chern class $c_1^{\Ver}$ of the vertical tangent bundle and
the coupling class $u_\ga$, the unique class that extends the
fiberwise
symplectic class $[\om]$ and is such that $u_\ga^{n+1}=0.$
Then 
\begin{equation}\labell{eq:S}
\Ss(\ga): = \sum_{\si} a_\si\otimes q^{-c_1^{\Ver}(\si)} 
t^{-u_{\ga}(\si)},
\end{equation}
where $\si\in H_2(P;\Z)$ runs over all section classes (i.e. classes that cover the positive generator of $H_2(\PP^1;\Z)$) and 
$a_\si\in H_*(M)$ is defined by the requirement that
\begin{equation}\labell{eq:Sa}
a_\si\cdot_M c = \bla c\bra^P_{\si},\quad \mbox{for all }c\in H_*(M).
\end{equation}
(Cf. \cite[Def.~11.4.1]{MS}.  For this to make sense  we must use  
$\La^{\univ}$ instead of $\La_{\om}$
 as explained in Remark~\ref{rmk:univ}. Note that we write
  $\cdot_M$ for the intersection product in $M$.)
Further for all $b,c\in H_*(M)$
\begin{equation}\labell{eq:Sa1}
\Ss(\ga)*b = \sum_{\si} b_\si\otimes q^{-c_1^{\Ver}(\si)} 
t^{-u_{\ga}(\si)}, \quad\mbox{where }\; b_\si\cdot_M c:=\bla b,c\bra^P_{\si}.
\end{equation}

\begin{rmk}\labell{rmk:horpt}\rm
Lemma~\ref{le:QH} shows that  if $(M,\om)$ is not strongly uniruled
then $\Ss(\ga) = \1\otimes \la + x$ where $x\in \Qq_-(M)$.  Therefore 
$\Ss(\ga)*pt = pt\otimes\la$, which in turn means that all $2$-point 
invariants $\bla pt,c\bra^P_{\si}$ with $c\in H_{<2n}(M)$ must vanish.
In other words, a section invariant in $P_\ga$ with more than a single point constraint must vanish. We shall expand on this theme later in 
Lemmas~\ref{le:horpt} 
and \ref{le:horpt2}.
\end{rmk}

In  general it is very hard  to calculate $\Ss(\ga)$.  
The following result is essentially due to Seidel~\cite[\S11]{Sei}; see also McDuff--Tolman~[Thm~1.10]\cite{MT}.  We shall say that
$K_t$ is a normalized
generating Hamiltonian for $\ga = \{\phi_t\}$ iff
$$
\om(\dot\phi_t,\cdot) = -dK_t, \;\;\mbox{ and }\; \int_M K_t\om^n = 0.
$$
Further we define
\begin{equation}\label{eq:Kmax}
 K_{\max}: = \int_0^1\max_{x\in M} K_t(x) dt.
\end{equation}

\begin{prop}\labell{prop:max}  Suppose that  $\ga$ has 
a nonempty maximal submanifold $F_{\max}$ and is generated by the normalized Hamiltonian $K_t$. Then 
$$
\Ss(\ga): = 
a_{\max}\otimes q^{m_{\max}} 
t^{K_{\max}} + \sum_{\be\in H_2(M;\Z),\;\om(\be)>0} a_\be\otimes q^{m_{\max}-c_1(\be)} 
t^{K_{\max}-\om(\be)},
$$
where $m_{\max}: = -c_1^{\Ver}(\si_{\max})$. Moreover
$a_{\max}\cdot_M c= \bla c\bra^P_{\si_{\max}}$.
\end{prop}
\begin{proof}  
Because $\ga$ is assumed to have a fixed maximum, we may write the section classes  $\si$ 
appearing in $\Ss(\ga)$  as $\si_{\max} + \be$
where $\be\in H_2(M)$, and then denote $a_\be: = a_\si$.  Therefore by
equation (\ref{eq:S}) the result will follow if we show that:\SSS

\NI
(a) 
{\it the only class with $\om(\be)\le 0$ that contributes to the 
sum is the class $\be = 0$}, and \SSS

\NI
(b) $-u_\ga(\si_{\max}) = K_{\max}$.\SSS

Consider the renormalized 
polar coordinates $(r,t)$ on the unit disc $D^2$, where $t: =  \theta/2\pi$. 
Define 
$\Om_-: = \om + \eps d(r^2) \wedge dt$ on $M\times D_-$
for some small constant $\eps>0$, and set
 $$
 \Om_+: = \om + \Bigl(\ka(r^2,t)d(r^2) - d\bigl(\rho(r^2) K_t\bigr)\Bigr)\wedge dt,\;\mbox { on } M\times D_+,
 $$
 where $\rho(r^2)$ is a nondecreasing function that equals $0$ near $0$ and $1$ near $1$.  If $ \ka(r^2,t) = \eps$ near $r=1$,
  these two forms fit together to give a closed form $\Om$ on $P$.
 Moreover, $\Om$ is symplectic iff $\ka(r^2,t) - \rho'(r^2) K_t(x)>0$ 
 for all $r,t$ and $x\in M$. Hence we may suppose that
\begin{equation}\labell{eq:nu}
 \nu: = \int_{s_{\max}}\Om = \pi\eps + \int_{D_+}
 \bigl(\ka(r^2,t) - \rho'(r^2) \max_{x\in M}K_t(x)\bigr) 2r dr dt
\end{equation}
 is as close to zero as we like. 
 
 A class $\be$ contributes to $\Ss(\ga)$ iff  
 some invariant $\bla c\bra^P_{\si_{\max}+\be}\ne 0$.  In particular,
 we must have $\int_{\si_{\max}+\be}\Om = \nu +\om(\be) > 0$ for all symplectic forms on $P$. Hence we need $\om(\be)\ge 0$.
 
 It remains to investigate the case
$\om(\be) = 0$.  
 Note that $\Om$ defines a connection on $P\to \PP^1$ whose horizontal spaces are
 the $\Om$-orthogonals to the vertical tangent spaces.  Further, this connection does not depend on the choice of function $\ka$.
 Choose an
$\Om$-compatible almost complex structure $J$ on $P$. We may assume that both the fibers of $P\to \PP^1$ and the horizontal subspaces are $J$-invariant.  Thus $J$ induces a compact family $J_z, z\in \PP^1$, of $\om$-tame almost complex structures on the fiber $M$.  
Let $\hbar$ 
be the minimum of the energies of all nonconstant 
$J_z$-holomorphic spheres in $M$ for $z\in \PP^1$.  
Standard compactness results imply that
$\hbar>0$.  Note that if we change $\ka$ (keeping $\Om$ symplectic)
$J$ is remains $\Om$-compatible.  Hence we may suppose that 
$\nu: = \int_{s_{\max}}\Om <\hbar.$

We now claim that  the only $J$-holomorphic sections of $P$ with energy $\le \nu $   are the constant sections 
$x\times \PP^1$ for $x\in F_{\max}$. 
 To see this, decompose tangent vectors to $P$  as $v+h$ where $v$ is tangent to the fiber and $h$ is horizontal.  Then, because of our choice of $J$, at a point $(x,z)$ over the fiber at $z = (r,t)\in D_+$,
\begin{eqnarray*}
\Om(v+h, Jv+Jh) &=& \om(v,Jv) + \Om(h,Jh)\;
\ge \;  
\Om(h,Jh) \\
&\ge& 2r\bigl(\ka(r^2,t) - \rho'(r^2) \max_{x\in M}K_t(x)\bigr) dr\wedge dt(h,Jh).
\end{eqnarray*}
Here the first  inequality is an equality only along a section that is everywhere horizontal, while the second is an equality only if the section  is contained in $F_{\max}\times \PP^1$. (Recall that by assumption
 $K_t(x)<K_{\max}$ for $x\notin F_{\max}$.)  Similarly, the corresponding integral over $D_-$  is $\ge \eps$ with equality only if
the section is constant.  Comparing with equation (\ref{eq:nu}) we see that the energy of a section is $\ge\nu$ with equality only if the section is constant.
This proves the claim. 

Now suppose that the class $\be$
 with $\om(\be) = 0$
contributes to $\Ss(\ga)$. Then the class $\si_{\max}+\be$ of energy $\nu$ must be represented by a $J$-holomorphic stable map consisting of a section,  possibly with some other bubble components in the fibers.  Since each bubble
 has energy $\ge \hbar>\nu$, the stable map has no bubbles and hence by the preceding paragraph must  consist of 
a single  constant section in the class $\si_{\max}$.  Thus 
 $\be = 0$. This completes the proof of (a).
 
 To prove (b) it suffices to check that the coupling class $u_\ga$ on
$P_\ga$ is represented by the form $\Om_0$ that equals $\om$ on $M\times D_-$ and
$$
\om  - d(\rho(r^2) K_t)\wedge dt,\;\mbox { on } M\times D_+.
$$
For this to hold we need that $\int_P\Om_0^{n+1}=0$. But this follows easily from the fact that $K_t$ is normalized. \end{proof}

\begin{cor}\labell{cor:S1}
Suppose that the fixed maximum $F_{\max}$ of $\ga$ is a divisor
and that   
 $\ga$ is an effective  circle action near $F_{\max}$.
 Then  
\begin{equation}\labell{eq:Ss}
 \Ss(\ga) = F_{\max}\otimes q\, t^{K_{\max}} + 
 \sum_{\be\in H_2(M;\Z),\,\om(\be)>0} a_\be\otimes q^{1-c_1(\be)}\,t^{K_{\max}-\om(\be)}.
\end{equation}
\end{cor}
\begin{proof}  By Proposition~\ref{prop:max}, it remains to check
 that the contribution $a_{\max}$ of the sections in class $\si_{\max}$ 
 to $ \Ss(\ga)$ is $F_{\max}$. 
 We saw there that
 the moduli space of (unparametrized) holomorphic sections in
 class $\si_{\max}$ can be identified with $F_{\max}$. In particular, it is compact.  Because $S^1$ acts with normal weight $-1$, $c_1^{\Ver}(\si_{\max}) =- m_{\max} = -1$ and
  one can easily see
 that  these sections are regular. (A proof is given in ~\cite[Lemma~3.2]{MT}.) Hence 
 $a_{\max}\cdot a: = \bla a\bra^P_{\si_{\max}} =  F_{\max}\cdot a$
 for all $a\in H_*(M)$.  
\end{proof}

If $(M,\om)$ is not strongly uniruled then, by Lemma~\ref{le:QH},  the unit
$\Ss(\ga)$ must be of the form $\1\otimes \la + x$ where $x\in H_{<2n}(M)$ and $\la\in \La$ is also a unit.  
The previous lemma shows that $\la = rt^{K_{\max} -\ka_0} + {\rm l.o.t.}$
where $r\in \Q$ is nonzero, $\ka_0=\om(\be)>0$, and l.o.t. denotes lower order terms.  We need to estimate the size of $\ka_0$
for the blow up $(\TM,\Tom)$.  In the course of the argument we 
shall also need to consider
 the bundles
$P'\to \PP^1$ and $\TP'\to \PP^1$ defined by $\ga': = \ga^{-1}$ and its blow up $(\Tga)^{-1}.$  We will represent the inverse loop 
$\ga^{-1}$ by the path $ \{\phi_{1-t}\}$, so that it 
is generated by the Hamiltonian $K_t':=-K_{1-t}$. If for any normalized Hamiltonian $K_t$ we set 
$
K_{\min}: =\int \min _{x\in M}K_t(x) dt,
$
then  
$$K_{\max}' = -K_{\min},\;\;\mbox{ and }\;
K_{\min}' = -K_{\max}.
$$ 
Similarly,  If $\TK_t$ and $\TK_t'$ generate the blow up loop $\Tga$ and $\Tga': = \Tga^{-1}$ then
$$
\TK_{\max}' = -\TK_{\min},\;\;\mbox{ and }\;
\TK_{\min}' = -\TK_{\max}.
$$

The above remarks imply that the coefficients $\la,\la'$  of 
$\1$  in the formulas 
for $\Ss(\ga)$, $\Ss(\ga')$ have the form
$$
\la = r_0t^{K_{\max}-\ka_0} + {\rm l.o.t.},\qquad
\la' = r_0't^{K'_{\max}-\ka_0'} + {\rm l.o.t.}
$$
for some nonzero $r_0, r_0'\in \Q$.  
Similarly, we write the
coefficients  $\Tla, \Tla'$ of 
$\1$  in the formulas 
for $\Ss(\Tga)$, $\Ss(\Tga')$ as
$$
\Tla = \Tr_0t^{\TK_{\max}-\Tka_0} + {\rm l.o.t.},\qquad 
\Tla' = \Tr_0't^{\TK_{\max}'-\Tka_0'} + {\rm l.o.t.}
$$
where  $\Tr_0, \Tr_0'$ are nonzero.
Corollary~\ref{cor:S1} implies that 
if $[\om]$ is assumed to be integral then $\ka_0\in \Z$ because it is a value $\om(\be)$ of $[\om]$ on an integral class $\be\in H_2(M;\Z)$.
Similarly $\Tka_0$ is a value of $\Tom$ on $H_2(\TM;\Z)$.
On the other hand, 
because the loop $\ga'$ need not have a fixed maximum
 $\ka_0'$ need not be a value of $\om$ on $H_2(\TM;\Z)$, although 
 $-K_{\max}' - \ka_0'$ is a value of the coupling class $u_{\ga'}$
  on $H_2(P';\Z)$.
 
In the next lemma we assume that $(\TM,\Tom)$ is not strongly uniruled.
Hu~\cite[Thm.~1.2]{Hu} showed that  $(M,\om)$ is strongly uniruled 
only if $(\TM,\Tom)$ is also (cf.  Lemma~\ref{le:strun}).  Hence
our hypothesis implies that     
 $(M,\om)$ also is not strongly uniruled. 
 
\begin{lemma}\labell{le:K2} Let $(\TM,\Tom)$ be the blow up of $(M,\om)$ at a point of $F_{\max}$.   Then:
 \SSS
 
 \NI{\rm (i)} $K_{\max} - K_{\min}=\TK_{\max} - \TK_{\min}$.\SSS

\NI {\rm (ii)} If $(\TM,\Tom)$ is not strongly uniruled then
$\ka_0 + \ka_0' =\Tka_0 + \Tka_0' = K_{\max} - K_{\min}$.\SSS

\NI {\rm (iii)}  If $(\TM,\Tom)$ is not strongly uniruled,
then   $\Tka_0 = \ka_0$ and $\Tka_0' = \ka_0'$.
\end{lemma}

\begin{proof}  Observe
 that the $S^1$ action near the point $x_{\max}\in F_{\max}$ in $M$ extends to a local toric structure in some neighborhood $U$ of 
$x_{\max}$ that preserves the submanifold $U\cap F_{\max}$.   Hence we can form $(\TM,\Tom)$  by cutting off a neighborhood of the vertex of the local moment polytope corresponding to $x_{\max}$.
Since we do not cut off the whole of $F_{\max}$ this does not change
the length $ \max_{x\in U}K_t(x) - \min_{x\in U}K_t(x)$ of the image of the normalized local $S^1$-moment map at each $t$.  This proves (i).

The identity
$$
\1=\Ss(\ga)*\Ss(\ga') = (\1\otimes \la + x)(\1\otimes \la' + x')
$$
implies that $\la'= \la^{-1}$. By Lemma~\ref{le:QH} this is possible only if
${K_{\max}-\ka_0} + {K_{\max}'-\ka_0'}=0$. Therefore
$\ka_0 + \ka_0' = K_{\max} - K_{\min}$. (Here we use the fact that
$K'_{\max} = -K_{\min}$.) A similar argument shows that
$\Tka_0 + \Tka_0' = \TK_{\max} - \TK_{\min}$. In view of (i),
this proves (ii).

The proof of (iii) is deferred to the end of this section (after Corollary~\ref{cor:SU}). 
\end{proof}
 

We now prove the key result of this section 
that ties the current ideas to the
previously developed algebra. 
 
\begin{prop}\labell{prop:SU} Suppose that $(M,\om)$ is not uniruled,
and that $[\om]\in H_*(M;\Z)$. Suppose further that
$\ga$ has a maximal fixed point set
$F_{\max}$  that is a  divisor, and that $\ga$ acts near $F_{\max}$ 
as an effective circle action.
Then
there is a unit $u\in \Ee_{2n}: = QH_{2n}(\TM)/\Ii$ of the form
$$
u = rE\otimes qt^{\ka_0} + \1 + x,\quad  r\ne 0,
$$
where $x = \sum_{\ka_i\ge \de} E^{j_i} \otimes q^{j_i} t^{\ka_0-\ka_i}$
for some $0<j_i<n$ and  $\ka_0\in \Ga_\om$ is positive.
\end{prop} 
\begin{proof}  Let $(\TM, \Tom)$ be the one point blow up of $(M,\om)$ with blow up parameter $\de$ as before. 
Then $[\TF_{\max}] = [F_{\max}]-E$ where $E$ is (the class of) the exceptional divisor and
we consider $[F_{\max}]\in H_{2n-2}(M)\equiv
H_{2n-2}(\TM\less E)$. Denote by $\TK_t$ the normalized generating Hamiltonian for the lift $\Tga$ of $\ga$ to $\TM$. 
By Corollary~\ref{cor:S1} the Seidel element of $\Tga$  has the form 
\begin{eqnarray}\labell{eq:Ss3}\notag
 \Ss(\Tga) &=& \TF_{\max}\otimes q\, t^{\TK_{\max}} + \sum_{\Tbe\in H_2(\TM;\Z),\,
 \Tom(\Tbe)>0} a_\be\otimes q^{1-c_1(\Tbe)}\,t^{\TK_{\max}-\Tom(\Tbe)}\\
 & = & -E\otimes q t^{\TK_{\max}} + \sum_i E^{j_i}\otimes q^{j_i}\,t^{\TK_{\max}-\Tka_i} \pmod \Ii.
\end{eqnarray}
Denote the coefficient of $\1= E^0$ in this formula by $\Tla\in \La^{\univ}$.  Since  $(\TM,\Tom)$ is the blow up of $(M,\om)$ it
is not uniruled by 
Hu--Li--Ruan~\cite[Thm~1.1]{HLR}.  Therefore the fact that
$\Ss(\Tga)$ is a unit implies by Lemma~\ref{le:QH} that $\Tla$
is also a unit.  Write $\Tla = t^{\TK_{\max} -\Tka_0}(\Tr_0 + y)$ where $\Tr_0\ne0$ and $y$ is a sum of negative powers of $t$.
Lemma~\ref{le:K2}(iii)  implies that $\Tka_0= \ka_0 =\om(\be)>0$ for  some $\be\in H_2(M)$.   Now let $u =  \Ss(\Tga)\otimes \Tla^{-1}$ mod $\Ii$. This has the required form.
\end{proof}

In the next corollary we denote the exceptional divisor in the fiber $\TM$
of $\TP'\to \PP^1$ by $E$, and write $E^j$ for its $j$-fold intersection product in $\TM$ considered (where appropriate) as a class in $\TP'$.

\begin{cor}\labell{cor:SU} (i)  Let $n\ge 3$.  Then under the conditions of 
Proposition~\ref{prop:SU}
there is $\si\in H_2(P';\Z)$  such that
$\bla   E^2 \bra^{\TP'}_{\si-\eps} \ne 0$.\SSS

\NI (ii)  Similarly, if $n=2$ there is 
$\si\in H_2(P';\Z)$  such that
$\bla   E \bra^{\TP'}_{\si-2\eps} \ne 0$.
\end{cor}
\begin{proof} (i)  By Lemma~\ref{le:U} 
$u^{-1}$ has the form
$1 - s^{n-1}t^\de + \frac 1r s^{n-2}t^{\de-\ka_0}\bigl(1+{\rm l.o.t.}\bigr)$.
But $u^{-1} =  \Ss(\Tga')\otimes \Tla$ mod $\Ii$. 
Therefore 
$$
\Ss(\Tga') = \Tla^{-1}\Bigl(1 - s^{n-1}t^\de + \frac 1r s^{n-2}t^{\de-\ka_0}\bigl(1+{\rm l.o.t.}\bigr)\Bigr) \pmod \Ii.
$$
By Lemma~\ref{le:K2} (iii)  $\Tla^{-1}$ is a series in $t$ with highest order term $t^{-\TK_{\max}+\ka_0}$. 
 Therefore the 
largest $\ka$ such that the
coefficient  of $s^{n-2} t^\ka$  in $\Ss(\Tga')$ is nonzero 
is 
$$
\ka = -\TK_{\max}+\ka_0+\de-\ka_0 = -\TK_{\max}+\de.
$$
  Since the coefficient of
$t^{K'_{\max} - \ka_0'}$ in $\Ss(\ga')$ is nonzero, it follows from
the definition of $\Ss$ in equation 
(\ref{eq:Ss}) that  there is a section $\si_0'$ of $P'$ such that
$-u_{\ga'}(\si_0') = K'_{\max} - \ka_0'$.  We will take this as our reference section, writing all other sections in $P'$ in the form $\si_0'+\be$ where $\be\in H_2(M;\Z)$ and those of $\TP'$ as $\si_0'+\Tbe$ where $\Tbe\in H_2(\TM;\Z)$.

Observe that
 $$
-u_{\Tga'}(\si_0') = \TK'_{\max} - \Tka_0' = 
\TK'_{\max} - \ka_0' = - \TK_{\max}+\ka_0, 
$$
where the second equality holds
 by Lemma~\ref{le:K2} (iii), and the third 
   by Lemma~\ref{le:K2} (ii) and the identity
$ \TK'_{\max}= -\TK_{\min}$.
Since the coefficient of $t^\ka$ in $\Ss(\Tga')$ is nonzero by hypothesis,
 there is $\Tbe \in H_2(\TM;\Z)$ such that
$$
\ka = -u_{\Tga'}(\si_0'+\Tbe) = -u_{\Tga'}(\si_0') - \Tom(\Tbe).
$$
Therefore  
$$
\Tom(\Tbe) = -u_{\Tga'}(\si_0')-\ka = 
 - \TK_{\max}+\ka_0 + \TK_{\max}- \de = \ka_0-\de.
 $$
  Since $I_{\Tom}$ is injective
  $\Tbe$ must have the form $\be -\eps$ for some $\be \in H_2(M)$. 
  Hence we may write the class $\si_0' + \Tbe$ as $\si_0'+\be -\eps = \si -\eps$ for some $\si=\si_0'+\be \in H_2(P')$.
  Therefore the coefficient of $s^{n-2}t^\ka$ in $\Ss(\Tga')$ arises from 
  a nonzero  invariant of the form
$$
\bla   E^{2} \bra^{\TP'}_{\si-\eps}
$$ where $\si \in H_2(P')$.
(To check the power of $E$ here, note that $E^2\cdot_M E^{n-2} = -1$.)
This proves the first statement.

The proof of (ii) is similar and is left to the reader.
\end{proof}

\begin{rmk}\rm  (i)  By construction $-u_{\ga'}(\si) =
-K_{\min}-\ka_0'-\ka_0 = -K_{\max} = -u_{\ga}(s_{\min}')$,
where $s_{\min}'$ is the section of $P'$ that is blown up to get $\TP'$.   Therefore the classes $\si$ and $\si_{\min}'$ are equal.\SSS

\NI (ii)   The reader might wonder why we ignored the 
coefficient of
$s^{n-1}t^\de$ in $\Ss(\Tla')$.  But, reasoning as above, 
one can check that this term gives rise to a nonzero invariant of the form
$$
\bla   E \bra^{\TP'}_{\si-\eps},
$$
 where $\si = \si'_0$.
When one blows $\TP'$ down to $P'$ this corresponds to 
the invariant $\bla  pt \bra^{P}_{\si}$, which we already know is nonzero. Hence this term  gives no new information.
\end{rmk}

We finish this section by proving Lemma~\ref{le:K2} (iii).  For this we need the following preliminary result.
 
\begin{lemma}\labell{le:P}  Let $P, \TP, P',\TP'$ be as above. Then:

\NI{\rm (i)}
For any  section class $\si\in H_2(P)$ 
$$
\bla pt\bra^P_\si = \bla pt\bra^\TP_\si,
$$
where on RHS we consider $\si\in H_2(\TP)$.
\SSS

\NI {\rm (ii)} Similarly, for any  section class $\si\in H_2(P')$ 
$$
\bla pt\bra^{P'}_\si = \bla pt\bra^{\TP'}_\si,
$$
where on RHS we consider $\si\in H_2(\TP')$.
\end{lemma}
\begin{proof} Hu showed  in~\cite[Thm.~1.5]{Hu} 
that if $\TQ$ is the  blow up 
of $Q$ along an embedded $2$-sphere  
$C$ in a class  with $c_1(C)\ge 0$ then
the Gromov-Witten invariants in classes $\si\in H_2(Q)$ with  homological constraints in $H_*(Q\less C)$  remain unchanged.  (The proof is similar to that of 
Proposition~\ref{prop:eps} below.) The result now follows because $\TP$ is the blow up of $P$ along the section $s_{\max}$ with $c_1(s_{\max}) = 1$, while $\TP'$ is the blow up of $P'$ along a section 
$s_{\min}'$ with $c_1(s_{\min}') = 3.$
 (Cf. \S\ref{ss:eps} for more detail on the construction of $\TP'$.)
\end{proof}

\NI {\bf Proof of Lemma~\ref{le:K2} (iii).}   
We need to show that $\ka_0 = \Tka_0$ and $ \Tka_0'=\ka_0'$. Since 
$\ka_0 + \ka_0' = \Tka_0 + \Tka_0'$ by part (ii) of
 Lemma~\ref{le:K2}, it suffices to show that
$\Tka_0\le\ka_0$ and $ \Tka_0'\le\ka_0'$. 

But by equations (\ref{eq:S}) and (\ref{eq:Sa}),  $\ka_0$ is the minimum of $\om(\be)$ over all classes $\si=\si_{\max}+\be$ such that $\bla pt\bra^P_{\si}\ne 0$.  Similarly,
$\Tka_0$ is the minimum of $\Tom(\Tbe)$ over all classes $\Tsi=\Tsi_{\max}+\Tbe$ such that $\bla pt\bra^P_{\Tsi}\ne 0$.
But $\si_{\max} = \Tsi_{\max}$, and
by Lemma~\ref{le:P}(i) 
every  class $\si$  with nonzero invariant in
$P$ also has nonzero invariant
in $\TP$.  Hence  we must have  
$\Tka_0\le\ka_0$.  

Now consider the fibrations $P'\to \PP^1, \TP'\to \PP^1$. 
To compare $\ka_0'$ with $\Tka_0'$ we need to use suitable reference sections that do not change under blow up.  Instead of the maximal sections $s_{\max}, \Ts_{\max}$ we use the corresponding minimal sections  $s_{\min}'$ and $\Ts_{\min}'$ that represent the classes
$\si_{\min}'$ and $\Tsi_{\min}'$.
Thus we write every section class
 $\si'$ of $P'$ as
$\si_{\min}' + \be$, where $\be\in H_2(M)$.  Note that $\si_{\min}'$ is taken to $\Tsi_{\min}'$ under the natural inclusion $H_2(P')\to H_2(\TP')$.

The symplectomorphism from 
the fiber connect sum $P\#P'$ to the trivial bundle takes
 the connect sum of the sections $s_{\max}$ and $s_{\min}'$ 
to the trivial section $\PP^1\times pt$ and the connect
sum of the coupling classes $u_{\ga}\#u_{\ga'}$ to the coupling class $pr_M^*(\om)$ of the trivial bundle $\PP_1\times M\stackrel{pr}\longrightarrow \PP^1$.
Hence
$$
u_{\ga}(s_{\max}) + u_{\ga'}(s_{\min}') = pr_M^*(\om)(\PP^1\times pt) = 0.
$$
Therefore $u_{\ga'}(s_{\min}') = -u_{\ga}(s_{\max}) =K_{\max}$.
A similar argument applied to the blowups shows that
$u_{\Tga'}(s_{\min}') = \TK_{\max}$.

By Proposition~\ref{prop:max} we can interpret 
$\ka_0'$   as the minimum of $K'_{\max} +u_{\ga'}(\si')$ over all
classes $\si'$ such that $\bla pt\bra^{P'}_{\si'}\ne 0$.  
Writing $\si'$ as $\si_{\min}' + \be$ and using
$ K'_{\max}=-K_{\min}$, we find
 $$
 K'_{\max} +u_{\ga'}(\si_{\min}' + \be) = K_{\max} -K_{\min} + \om(\be).
 $$
 Hence $\ka_0'$  is the minimum of $K_{\max} -K_{\min} +\om(\be)$ over the corresponding set $\Bb$ of classes $\be$. Similarly, $\Tka_0'$ is
 the minimum of $\TK_{\max} -\TK_{\min} +\Tom(\Tbe)$ over a set of classes $\Tbe$ that, by 
 Lemma~\ref{le:P} (ii) and the fact that $\si_{\min}' = \Tsi_{\min}'$,  
  includes $\Bb$.
 Since $K_{\max} -K_{\min} =\TK_{\max} -\TK_{\min}$, we find that
$\Tka_0'\le \ka_0'$ as required.
\QED

\subsection{Proof of the main results}

Suppose that $\ga$  is a loop of Hamiltonian 
symplectomorphisms of $(M,\om)$ with a fixed maximum near which it is an effective circle action.
A standard Moser argument\footnote
{Perturb the given form to a suitable class and then average over the $S^1$ action.}
 implies that  we can assume that $I_\om$ is injective.
The next lemma shows that we can always arrange for the 
other conditions of Proposition \ref{prop:SU} to hold.

\begin{lemma} \labell{le:semi} 
Suppose that the submanifold $F_{\max}$ of $M$  is a fixed maximum of the loop $\ga$ and that nearby $\ga$ is an effective $S^1$ action. 
Then we can blow up along  $F_{\max}$  to achieve an action on
some blow up of $M$ for which the new $F_{\max}$ is a  divisor
with normal weight $m=1$.  
\end{lemma}
\begin{proof}  Suppose that the weights of the action normal to $F_{\max}$ are
$(-m_1,\dots,-m_k)$ where $0<m_1\le m_2\le \dots\le m_k$. Let $\TM$ be the blow up of $M$ along $F_{\max}$.  Then the weights of the induced action
normal to the new maximal fixed point set are $-m_1, -(m_i-m_1), \dots, -(m_k-m_1)$, where $i$ is the smallest index such that $m_i> m_1$.  
Therefore by repeated blowing up along the maximal fixed point of the moment we can reduce to the case when there is just one normal weight, i.e. $F_{\max}$ is a divisor.  The normal weight is $1$ since the action is effective. \end{proof}

Corollary~\ref{cor:SU} summarizes the information given to us by the Seidel representation.   To make use of it,
we need one more result relating the 
Gromov--Witten invariants of $\TP'$ and $P'$.  This is proved in 
\S \ref{ss:eps}.

\begin{prop}\labell{prop:eps} Suppose that $(M,\om)$ is not uniruled.
\SSS

\NI (i) If $n\ge 3$ and
$
\bla   E^2 \bra^{\TP'}_{\si-\eps} \ne 0
$
for some 
section class $\si\in H_2(P';\Z)$ then
 $\bla \tau_{1} pt\bra^{P'}_{\si} \ne 0$.\SSS

\NI (ii) If $n=2$ and $\bla E\bra^{\TP'}_{\si-2\eps}\ne 0$ then
for some section $s$ in $P'$  at least one of $\bla  pt,s\bra^{P'}_{\si}$ and $\bla  \tau_{1} pt\bra^{P'}_{\si}$
is nonzero.
\end{prop}

We next show that the situation discussed in the above proposition cannot occur: in fact, if the conclusions hold $(M,\om)$ must be strongly uniruled.

\begin{lemma}\labell{le:horpt} If there is an element 
$\ga\in \pi_1(\Ham(M))$ such that the descendent invariant
$$
\bla \tau_k pt\bra^P_{\si}\ne 0
$$
for some $k>0$ and some section class $\si$ in $P: = P_\ga$ then
$(M,\om)$ is  strongly uniruled.
\end{lemma}
\begin{proof}   As explained in \S4 (see equation~(\ref{eq:LPi}))
we may express this invariant as a sum
$$
\bla \tau_k pt,M,M \bra^P_{\si} = \sum_{i,\al_1+\al_2=\si} \bla \tau_{k-1} pt,\xi_i\bra^P_{\al_1}
\bla \xi_i^*, M,M\bra^P_{\al_2},
$$
where $\al_j\in H_2(P)$ and $\xi_i$ runs over a basis for $H_*(P)$ with dual basis $\xi_i^*$.    Because the symplectic manifold $(P,\Om)$ supports
an $\Om$-tame almost complex structure $J$ such that 
the projection $P\to \PP^1$ is holomorphic, the only classes in $H_2(P)$ with nontrivial GW invariants project to nonnegative multiples of the generator of $H_2(\PP^1)$.  Thus one of the classes $\al_j$ is a section class and the other is a fiber class, i.e. a class in $H_2(M)$.  But if $\al_2$ is a fiber class then
it cannot meet two separate copies of $M$.  Hence any nonzero term in this expansion must have section class $\al_2$.  Then $\al_1$ is a fiber class,
and so by Lemma~\ref{le:fib}
$$
\bla \tau_{k-1} pt,\xi_i\bra^P_{\al_1}=\bla \tau_{k-1} pt,\xi_i\cap M\bra^M_{\al_1}.
$$
But now by repeated applications of equation~(\ref{eq:LPi})
one finds that $M$ is strongly uniruled; cf.
the proof of Lemma~\ref{le:GWsu}.
\end{proof}

The following lemma is proved by a similar argument; see \S\ref{ss:LP} at the end of \S\ref{s:rel}.

\begin{lemma}\labell{le:horpt2}  Suppose that $(M,\om)$ is the blow up of
a symplectic $4$-manifold that is not rational or ruled.  Then there is no
 $\ga\in \pi_1(\Ham(M))$ such that the corresponding fibration $P\to \PP^1$
has a section $s$ and a section class $\si$ with
$$
\bla pt,s\bra^P_\si \ne 0.
$$
\end{lemma}

\NI {\bf Proof of Theorem~\ref{thm:main}.}  
Suppose first that $\dim M\ge 6$ and that $(M,\om)$ is not uniruled. By
 Lemma~\ref{le:semi} we may  assume that we are in the situation of Proposition~\ref{prop:SU}.  Then Corollary~\ref{cor:SU}(i) together with Lemma~\ref{le:horpt} implies that $(M,\om)$ is strongly uniruled.
 This contradiction  shows that the initial hypothesis must be wrong: 
 in other words,  $(M,\om)$ is uniruled.
 
 When $\dim M=4$ the argument is similar, except that we use
  Lemma~\ref{le:horpt2} instead of
Lemma~\ref{le:horpt}.
 \QED

Finally note that Proposition~\ref{prop:loop}
follows from the next lemma.

\begin{lemma}\labell{le:pos} Suppose that the loop $\ga$ in $\Ham(M,\om)$  has a nondegenerate fixed maximum at
$x_{\max}$.  Suppose also that 
 the linearized flow $A_t, t\in [0,1],$ at $x_{\max}$ is 
homotopic through positive paths to a linear circle action. 
Then $\ga$ is homotopic through Hamiltonian loops  
with fixed maximum at $x_{\max}$ to a loop $\ga'$
 that is a circle action near $x_{\max}$. 
\end{lemma}
\begin{proof} Identify a neighborhood of $x_{\max}$ with 
a neighborhood of $\{0\}$ in $\R^{2n}$ with its standard symplectic form $\om_0$.\MS

\NI {\bf Step 1:}  {\it  The loop $\ga=:\ga_0$ is homotopic
 through loops $\ga_s = \{\phi_{st}\}$
with fixed maximum at  $\{0\}$ to a loop $\ga_1$ that 
equals the linear flow $A_t$ in some neighborhood of  $\{0\}$.}\MS
  
  Let $\ga = \{\phi_t\}$ with generating Hamiltonian $H_t$.
  By assumption  $H_t = H^0_t + O(\|x\|^3)$ where for each $t$ 
  there is a positive definite symmetric matrix $Q_t$
  such that $H_t^0(x) = -x^TQ_t x$.
We shall choose $\phi_{st}$ of the form  $\phi_t\circ g_{st}\circ h_{st}$ 
where \SSS

\NI
$\bullet$ for each $s\in [0,1]$, the loop $g_{st}, t\in [0,1],$ 
consists of 
diffeomorphisms with support in 
some small  
ball $B_{2r_1}: = B_{2r_1}(0)$ such that for all $s,t$ we have
$$ 
dg_{st}(0)=id, 
\quad g_{0t}=id ,\quad g_{1t} = (\phi_t)^{-1}\circ A_t\mbox{ in  }B_{r_1};
$$

\NI
$\bullet$  $h_{st}$ has support in the ball $B_{3r_1}$ and is such that $h_{st}^*(g_{st}^*\om_0) = \om_0$. Moreover $h_{1t}=id$ on $ B_{r_1}$.\SSS

Choose 
$g_{st}$ satisfying all conditions above and with support in ${\rm int\,}
B_{2r_1}$.  Since   $dg_{st}(0)=id$, we may write 
$g_{st}(x) = x + O(\|x^2\|)$.
Hence  we may 
arrange  that 
for some constants
$c_0, c_1$   we have $\|g_{st}(x)-x\|\le c_0\|x\|^2$ and
$|g_{st}^*\om_0-\om_0|\le c_1\|x\|$.  
Note that  we may assume that these constants $c_i$, as well as all subsequent ones, 
 depend only on the initial path $\phi_t$, i.e. 
 there are constants $R>0$ and $c_i$ such that suitable $g_{st}$ exist
 for all $r_1<R$.  Then, reducing $R$ as necessary, $c_1\|x\|$ is arbitrarily small on $B_{2r_1}$ so that we may assume that the
forms $\om_{st}: =  g_{st}^*\om_0$ are all
nondegenerate.  Let $\rho_{st}: = \frac d{dt}(g_{st}^*\om_0)$.
Since $g_{st} = id$ for all $s$ near $0$ by construction,
$\rho_{0t}=0$. Similarly, $\rho_{1t}=0$ in $B_{r_1}$.  Further for some constant $c_2$ as above
 $|\rho_{st}|\le c_2\|x\|$.

Now,  $\rho_{st}: =  d\be_{st}$, where for each $s,t$
 $$
 \be_{st}(x) = \int_0^{1} \rho_{st}(\p_r,\,\cdot\,)(\la x)d\la,\quad x\in B_{2r_1},
 $$
and $\p_r$ is the radial vector field in $\R^{2n}$.
 Hence $\be_{0t} \equiv 0$, and $\be_{1t}=0$ in $B_{r_1}$.  Moreover,
 each $1$-form $\be_{st}(x)$ satisfies $|\be_{st}(x)|\le c_3\|x\|^2$ and,
 because $\rho_{st}$ has support in ${\rm int\,} B_{2r_1}$,
  is closed and independent of $r: = \|x\|$ near $\p B_{2r_1}$.  Therefore, near $\p B_{2r_1}$ we may write
 $\be_{st} = df_{st}$ where $f_{st}$ is a function on $S^{2n-1}$ 
 of norm $\le c_4r_1^3$.  Therefore we may extend $\be_{st}$ by $d(\al(r)f_{st})$ for a suitable cut off function $\al$ so that it has support in $B_{3r_1}$ and still satisfies an  estimate 
 $|\be_{st}(x)|\le c_5\|x\|^2$.  
 
Now construct $h_{st}, t\in [0,1],$ for each fixed $s$ 
 by the usual  Moser homotopy method so that  $\om_{st} (\frac d{dt}h_{st},\cdot) = \be_{st}$.
 Then $h_{st}$ satisfies the required conditions.
 Moreover
$\|\frac d{dt} h_{st}\|\le c_6\|x\|^2$. 
Hence the Hamiltonian $G_{st}$ that generates $g_{st}\circ h_{st}$ satisfies
 $$
 |G_{st}| \le c_7\|x\|^3 \;\;\mbox{for }\|x\|\le 3r_1.
 $$
 The generating Hamiltonian for $\phi_{st}: =  \phi_t\circ g_{st}\circ h_{st}$
 is $$
 H_{st}(x): = H_t(x) + G_{st}(\phi_t^{-1}(x)) = H^0_t(x) + O(\|x\|^3).
 $$
  Since $c_7$ is independent of $r_1$, we can now choose $r_1$ so small that 
  $x=0$ is still the global maximum of $H_{st}$.
(The size of $r_1$ will depend on 
   the smallest eigenvalues of the matrices $Q_t$
   and also on the second derivatives of the $\phi_t$, i.e. the cubic term in $H_t$.)
This completes Step 1.\SSS
 
 By assumption there is 
 a homotopy
 of the path $\{A_t\}: = \{A_{1t}\}$ to a 
circle action $\{A_{2 t}\}$ through positive paths 
$\{A_{st}\}_{t\in [0,1]}$.  
Denote by $Q_{st}$ the corresponding  family of positive definite
matrices.
Reparametrize this homotopy with respect to $s$ to stretch it out over the interval $s\in [1,N]$ for some large $N$ to be chosen later.
 Note that $\phi_{1t} = A_{1t}$ on $B_{r_1}$.
\MS

 \NI {\bf Step 2:} {\it There is a sequence $r_k, k\ge 2,$ 
 satisfying $0<2r_{k+1}<r_k$ for all $k$
 and a finite sequence of homotopies $\ga_{s}=\{\phi_{st}\}, s\in [k,k+1], k=1,\dots,N-1,$ of Hamiltonian loops with support in
$B_{r_{k}}$
so that:\SSS
 
\NI
$\bullet$  for all 
 $k\in [1,N-1]$ 
 $\phi_{k+1\,t}= A_{k+1\,t}$ on $B_{r_{k+1}}$ and
 $ \phi_{k+1\,t}= A_{k\,t}$ on $B_{r_{k}} \less B_{2r_{k+1}}$.}
 \SSS

\NI
$\bullet$ {\it for each $s$ the Hamiltonian loop $\phi_{st}, t\in [0,1]$, 
has  fixed maximum at  $\{0\}$.}\SSS

 \NI  
 If $r_{k}$ and $ \phi_{k\,t}$ are  given,  an obvious modification of the  procedure described in Step 1 gives suitable
 $r_{k+1}$ and $ \phi_{k+1\,t}$ provided that the derivative
 $\frac d {ds} A_{st}, s\in [k,k+1],$ is not too large.
 As before, the idea is first to use the linear homotopy $A_{st}$ for $ s\in [k,k+1]$ to construct a smooth interpolation $g_{st}$ between $A_{k+1\,t}$ on 
 $B_{r_{k+1}}$ and $A_{k\,t}$ on a neighborhood of $\p B_{r_k}$, and then correct using a Moser homotopy.   Then $\|g_{st}(x) - A_{k\, t}(x)\|\le c\|x\|^2$, where  $c$ depends only on  $\frac d {ds} A_{st}.$
(Note that $c$ does not depend on $r_k$ because the linear maps $A_{st}$ are invariant under homotheties.)
   But we can arrange that $\frac d {ds} A_{st}$ is as small as we like  by choosing $N$ sufficiently large. In order that the resulting 
   loops $\phi_{st}$ have fixed maximum at  $\{0\}$,  the permissible size for $c$ (and hence  $N$)
will  depend on
the smallest eigenvalue of the compact family of matrices $Q_{st}$.

Step 2 completes the proof, since
the loop $\phi': = \phi_{Nt}$ satisfies the requirements of the lemma.
\end{proof}

\section{Gromov--Witten invariants and blowing up}\labell{s:rel}

We now prove the results on Gromov--Witten invariants needed above, i.e. Propositions \ref{prop:unicond2} and  \ref{prop:eps}.
 In \cite[Theorem~5.15]{HLR} Hu--Li--Ruan establish a correspondence 
between the
relative genus $g$ invariants of the blow up $(\TM,E)$ and certain
corresponding sets of absolute invariants of $M$. 
Proposition~\ref{prop:unicond2}
 could be proved using a special case
 of this general correspondence. However, instead of quoting their result we shall reprove the parts we need, since 
we do not need the full force of their results and also need 
some other related results.

We begin this section
with a discussion of relative GW invariants since this will be our main tool.  For more details see  Li--Ruan~\cite[\S4,5]{LR} and 
Hu--Li--Ruan~\cite[\S3]{HLR}, and Bourgeois {\it et al.}~\cite[\S10]{BE}
for relevant compactness results.  After stating the decomposition formula, we prove Proposition~\ref{prop:unicond}, a 
generalization of Proposition~\ref{prop:unicond2} that
characterizes uniruled manifolds $M$ in terms of properties of 
the one point blow up $\TM$.

Finally we prove Proposition~\ref{prop:eps} by using the decomposition formula for section classes of the bundle $P\to \PP^1$.

\subsection{Relative  invariants of genus zero }\labell{ss:rel}

Consider a pair $(X,D)$ where $D$ is a divisor in $X$, i.e. 
 a codimension $2$ symplectic submanifold.
The relative invariants count connected $J$-holomorphic curves in the $k$-fold prolongation $X_k$ (defined below) of $X$. 
Here $J$ is an $\om$-tame almost complex structure on $X$ satisfying certain normalization conditions along $D$.  In particular $D$ is $J$-holomorphic, i.e. $J(TD)\subset TD$. The invariants are defined by first 
constructing a compact moduli space $\oMm: = \oMm\,\!^{X,D}_{\be,\un d}(J)$ of genus zero $J$-holomorphic
curves $C$ in class $\be\in H_2(X)$ as described below and then integrating the given constraints over the corresponding virtual cycle $\oMm\,\!^{\,[vir]}$. Here $\un d = (d_1,\dots,d_r)$ is a partition of $d: =
 \be\cdot D\ge 0$.  
 
 To a first approximation the moduli space consists of curves in $X$, i.e. equivalence classes $C=[\Si,u,\dots]$ of stable maps to $X$, that intersect the divisor $D$ at $r$ points with multiplicities $d_i$.  
More precisely, each such  curve has $(k+1)$ 
levels $C_i$ for some $k\ge 0$, the principal level 
$C_0$  in $X\less D$ and the 
higher levels $C_i$ (sometimes called bubbles) in  the $\C^*$-bundle
$L_D^*\less D$ where $L_D^*\to D$ is the dual of the normal bundle to $D$. The whole curve $C$ therefore lies in a space $X_k$ called the $k$th prolongation of $X$, which is defined as follows.
Identify $L_D^*\less D$ with the complement of the sections $D_0,D_\infty$ of the ruled manifold 
$$
\pi_Q: Q: = \PP(\C\oplus L_D^*)\to D,
$$
where the zero section $D_0: = \PP(\C\oplus\{0\})$ has normal bundle $L_D^*$ and  the infinity 
section $D_\infty: = \PP(\{0\}\oplus L_D^*)$ has normal bundle $L_D$.
Think of the initial divisor $D\subset X$ as the infinity section $D_{0\infty}$ at level $0$. Then  
the  prolongation $X_k$ is simply the disjoint
 union of $X$ with $k$ copies 
 of $Q$, but it is useful to think that for each $i\ge 1$ the zero section $D_{i0}$ of the $i$th level is identified with the infinity section $D_{i-1\,\infty}$  of the preceding level.  The most important divisor in $X_k$ is 
 $D_{k\infty}$, which carries the relative constraints
 and hence plays the role of the relative divisor.

For each $i>0$ we assume that
the  ends of the components of the $i$th level curve  $C_i$  along the zero section $D_{i0}$  match with those of $C_{i-1}$ along $D_{i-1\,\infty}$. The final level $C_k$ carries the relative marked points
which are mapped to $D_{k\infty}$.
Assuming there are no absolute marked points, each level of $C$ is an equivalence class of stable maps 
 $$
 [\Si_i,u_i, y_{10},\dots,y_{r_{i0}0},y_{1\infty},\dots,
 y_{r_{i\infty }\infty}],
 $$
in some class $\be_i$,  where the internal relative marked points
 $y_{10},\dots,y_{r_{i0}0}$ are mapped to $D_{i0}$ 
 with multiplicities
  $\un m_{i0}$  that sum to $\be_i\cdot D_{i0}$
  and the marked points $y_{1\infty },\dots,y_{r_{i\infty }\infty}$ are taken to $D_{i\infty}$ with multiplicities $\un m_{i\infty }$ that sum to $\be_i\cdot D_{i\infty }$.  Note that the components $C_i$ and $C_{i+1}$ match along $D_{i\infty } = 
D_{i+1,0}$ only if the multiplicities $\un m_{i\infty }$ and
$\un m_{i+1,0}$ agree. In symplectic field theory such multilevel curves are called buildings: see~\cite[\S7]{BE}.

Each curve $C$ in $\oMm$ might also have some absolute marked points; these could lie on any of the levels $C_i$ but must be disjoint from the relative marked points.  
We require further that each component $C_i$   be stable, i.e. have a finite group of automorphisms.  For $C_0$ this has the usual meaning.  However, when 
  $i\ge 1$ we identify two level $i$ curves $C_i, C_i'$ if they lie in the same orbit of the fiberwise $\C^*$ action; i.e.  given representing maps $u_i:(\Si_i,j)\to Q$
 and 
$u_i':(\Si_i',j')\to Q$,  the curves are identified if there is $c\in \C^*$ and a holomorphic map $h:(\Si_i,j)\to (\Si_i',j')$ (preserving all marked points) such that $u_i\circ h=c \,u_i'$.\footnote{
Hence a level $C_i$ cannot consist only of multiply covered curves
$z\mapsto z^k$ in the fibers $\C^*$ of $L_D^*\less D$ because these are not stable.} 
Thus $L_D^*\less D$ should be thought of as a  rubber space.

  Note that although the whole curve is connected, 
the individual levels need not be but
should fit together to form a genus zero curve.  
The homology class $\be$ of such a curve is defined to be the sum of the homology class of its principal component with
the projections to $D$ of the classes of its higher levels.
\footnote
{
This is the class of the curve in $X$ obtained by gluing all the levels together.
When $k=0$ the curve only has one level and so the relative constraints lie along $D\subset X$.   Ionel--Parker~\cite{IP} work with the moduli space obtained by closing the space of $1$-level curves, which in principle could give slightly different invariants from the ones considered here. }  

When doing the analysis it is best to think that the domains and targets of the curves have cylindrical ends.  However, their indices
are the same as those of the corresponding compactified curves; see \cite[Prop~5.3]{LR}.  
 Thus the (complex) dimension of the moduli space $\oMm^{X,D}_{\be,k,(d_1,\dots,d_r)}$ of genus zero curves in class $\be$ with $k$ absolute marked points and $r$ relative ones is
\begin{equation}\labell{eq:dim}
 n + c_1^X(\be) + k+r-3 - \sum_{i=1}^r(d_i-1) = 
  n + c_1^X(\be) + k+2r-3 - d.
\end{equation}
Here $k+r-3$ is the contribution from moving the marked points and
we subtract $d_i-1$
at each relative intersection point of multiplicity $d_i$ 
 since, as far as a dimensional count is concerned, what is happened at such a point is that $d_i$ of the $d: = \be\cdot D$ intersection points of the $\be$-curve with $D$ coincide.

\begin{example}\labell{ex:rel}\rm   Let $X = \PP^2\#{\ov \PP}\,\!^2$, the one point blow up of $\PP^2$ and set $D: = E$, the exceptional curve. Denote by $\pi:X\to \PP^1$ the projection, and fix another section $H$ of $\pi$ that is disjoint from $E$. Let $J$ be the usual complex structure and $\oMm\,\!^{X,E}_\la(J;p)$ be the moduli space of holomorphic lines through some point $p\in H$, where $\la=[H]$ 
is the class of a line.
Since $\la\cdot E=0$ there are no relative constraints.
Then $\oMm\,\!^{X,E}_\la(J;p)$ has complex dimension $1$ and should be diffeomorphic to $\PP^1$.  The closure of the ordinary moduli space 
of lines in $X$ though $p$ contains
all such lines together with one reducible curve consisting of the union of the exceptional divisor $E$ with the fiber $\pi^{-1}(\pi(p))$ through $p$.  But the elements of $\oMm\,\!^{X,E}_\la(J;p)$ do not contain components in $E$.  Instead, this component becomes a higher level curve lying in
$Q=\PP(\C\oplus \Oo(1))$ that intersects 
$E_0=\PP(\C\oplus \{0\})$ in the point $E_0\cap \pi^{-1}(\pi(p))$ and lies in the class $\la_Q$ of the line in $Q\cong X$.  Note that modulo the action of $\C^*$ on $Q$ there is a {\it unique} such bubble.  Thus the corresponding two-level curve in $\oMm\,\!^{X,E}_\la(J;p)$ is a rigid object.
Moreover, because $\la_Q$ projects to the class $\eps\in H_2(E)$ of the exceptional divisor, its homology class $(\la -\eps) + pr(\la_Q)$ is $
(\la -\eps) + \eps = \la.$
\end{example}

The constraints for the relative invariants consist of homology classes  $b_j$ in the
divisor $D$ (the relative insertions) together with absolute (possibly descendent) insertions 
$\tau_{i_j}a_j$ where $a_j\in H_*(X)$ and $i_j\ge 0$.
   We shall denote
the (connected) relative genus zero invariants by:
\begin{equation}\labell{eq:GW}
\bla \tau_{i_1}a_1,\dots,\tau_{i_q}a_q\,|\, b_1,\dots,b_r\bra^{X,D}_{\be, (d_1,\dots,d_r)},
\end{equation}
where $ a_i\in H_*(X), b_i\in H_*(D),$ and $ d: = \sum d_i = \be\cdot D\ge 0.
$  (If $\be\cdot D< 0$ then the moduli space is undefined and the corresponding invariants are set equal to zero.)
This invariant counts isolated connected genus zero curves in class $\be$
that intersect $D$ to order $\un d: = (d_1,\dots,d_r)$ in the $b_i$,  i.e the $i$th relative marked point intersects $D$ to order $d_i\ge 1$ at a point on some representing cycle for $b_i$.  Moreover, the insertion
$\tau_{i}$ occurring at the $j$th absolute marked point $z_j$ means that
we add the constraint $(c_j)^{i}$, where $c_j$ is the first Chern class of the cotangent bundle $\Ll_j$ to the domain\footnote
{
Thus,  if  $z_j$ lies at the $m$th level, the fiber of $\Ll_j$ at $C$ is $T^*_{z_j}(\Si_m)$.}
 at $z_j$.    One can evaluate
(\ref{eq:GW}) (which in general is a rational number) by integrating
an appropriate product of Chern classes over the 
virtual cycle corresponding to
the moduli space of stable $J$-holomorphic  maps that satisfy the given homological constraints and tangency conditions.  
This cut down virtual cycle has dimension equal to the index of the curves $C$
satisfying these incidence conditions; if this dimension 
does not equal the total degree
$\sum_{j=1}^q i_j$ of the descendent classes the invariant is by 
definition set equal to $0$.
For details on how to construct this virtual cycle see for example Li--Ruan~\cite{LR} or Hu--Li--Ruan~\cite{HLR}. (Also cf. Remark~\ref{rmk:tech}.)

We shall need the following information about specific genus zero relative invariants. 

\begin{lemma}\labell{le:0}  
 Let $X = \PP^n$,  and  $D=\PP^{n-1}$ be the hyperplane. Denote by $\la$ the class of a line.
 Then:\MS
 
 \NI {\rm (i)}  
if $d>0$ and $n>1$
$$
\bla \;| b_1,\dots,b_r\bra^{\PP^n,\PP^{n-1}}_{0,d\la,\un d} = 0,\qquad \mbox{for any }
b_i\in H_*(\PP^{n-1}).
$$ 
{\rm (ii)} {\rm (Hu~\cite[\S3]{Hu} and Gathmann~\cite[Lemma~2.2]{Ga})} Let  $\TM$ be a one point blow up with
 exceptional divisor $E$ and suppose that $\be\in H_2(M)\subset H_2(\TM)$ so that $\be\cdot E = 0$.  
Then for any $a_i\in H_*(\TM\less E)$, the relative invariant $\bla a_1,\dots,a_k|\;\bra^{\TM,E}_{0,\be}$ equals the absolute invariant 
$\bla a_1,\dots,a_k\bra^{\TM}_{0,\be}$.
\SSS

\NI
{\rm (iii)} {\rm (Hu--Li--Ruan~\cite[Theorem~6.1]{HLR})} The $2$-point invariant
\begin{equation}\labell{eq:Eblow3}
 \bla \tau_{k}pt \,|\, D^j\bra^{\PP^n,\PP^{n-1}}_{0,d f, (d)},\qquad 1\le j\le n,\;d\ge 1,
 \end{equation}
is nonzero iff $k = nd-j$.
\end{lemma}
\begin{proof} (i) holds for dimensional reasons.  
We shall show that under the given conditions on $n$ and $d$  it is impossible to choose  $\un d$ and the $b_i$ so  that
$\oMm\,\!^{\,\PP^n,\PP^{n-1}}_{d\la,\un d}$ has formal dimension $0$.
Therefore the invariant vanishes by definition.

By equation~(\ref{eq:dim}) the formal  (complex) dimension of the  moduli space of genus zero stable maps through the relative constraints $b_1,\dots,b_r$ is 
$$
n + d(n+1) + r-3 - (d-r) + \de - rn,
$$
where $\de$ is half the sum of the degrees of the
 $b_i$.  
   Thus $0\le \de \le rn$.  
 Since $d-r\ge 0$ and $r>0$, we therefore need
$(d-r)n + n + 2r+\de= 3$, which implies $d=r$ and $n=r=1$. 
Since we assumed $n>1$, this is impossible.  

Hu's proof of (ii) uses the decomposition formula stated below.  
(His proof can be reconstructed by arguing as in the proof of Proposition~\ref{prop:unicond}.) 
The proof by Gathmann is more elementary but applies 
only in the case of projective algebraic manifolds.
Claim (iii) is much deeper.  It is proved by Hu--Li--Ruan   using localization techniques.
\end{proof}

\subsection{Applications of the decomposition formula}\labell{s:GWii}

Our  main tool  is the decomposition rule
of  Li--Ruan~\cite[Thm~5.7]{LR} and (in a slightly different version) of Ionel--Parker~\cite{IP}.  
So suppose that the manifold $M$ is the fiber sum of $(X,D)$ with $(Y,D^+)$, where the divisors
 $D: = D^-$ and $ D^+$ are symplectomorphic with dual normal bundles. 
 Since this is the only case we shall need, let us assume 
   that the 
absolute constraints can be represented by cycles in $M$ 
that do not intersect the inverse image of the divisor, i.e. that
 $a_i\in H_*(X\less D), i\le q,$ and $a_i\in H_*(Y\less D^+), q<i\le p$.   
 For simplicity we assume also that the map $H_2(M)\to H_2(X\cup_DY)$ is injective.  (This hypothesis is satisfied whenever $H_1(D)=0$.)  
   Further, let $b_i, i\in I,$ be a basis for $H_*(D)= H_*(D;\Q)$ with
dual basis $b_i^*$ for $H_*(D)$.  Then
the genus zero decomposition formula 
has the following shape:
\begin{eqnarray}\labell{eq:DF}
&&\bla a_1,\dots,a_p\bra^M_\be\; =\\\notag
&&\;  \sum_{\Ga,\un d,({i_1},\dots,{i_r})}  
n_{\Ga,\un d} \bla a_1,\dots,a_q\,|\, b_{i_1},\dots,b_{i_r}\bra^{\Ga_1,X,D}_{\be_1, \un d}
\bla a_{q+1},\dots,a_p\,|\, b_{i_1}^*,\dots,b_{i_r}^*\bra^{\Ga_2,Y,D^+}_{\be_2, \un d}. 
\end{eqnarray}
Here we  sum with rational weights $n_{\Ga,\un d}$ over all decompositions $\un d$ of $d$,  all possible
 connected labelled trees $\Ga$, and all possible sets 
 ${i_1},\dots,{i_r}$ of relative constraints.  
Each $\Ga$ describes a possible combinatorial structure for
 a stable map that glues to give
a $\be$-curve.  Thus
$\Ga$ is a disjoint union $\Ga_1\cup \Ga_2$, where the graph $\Ga_1$ (resp. $\Ga_2$) describes the part of the curve lying in some $X_k$ (resp. some $Y_\ell$). Also $\be_1$ (resp. $\be_2$) is the part of its label that describes the  homology class;  the pair $(\be_1,\be_2)$
runs through all decompositions such that the result of gluing the two curves in the prolongations 
$X_{k_1}$ and $Y_{k_2}$ 
along their intersections with the relative divisors 
gives a curve in class $\be.$
Moreover, there is a bijection between the labels $\{(d_i,b_i)\in \NN\times H_*(D)\}$
of the relative constraints in $\Ga_1$ and those
$\{(d_i,b_i^*)\in \NN\times H_*(D^+)\}$ in $\Ga_2$. (These labels are  called
relative \lq\lq tails"  in \cite{HLR}.)
 $\Ga_i$ need not be connected; if it is not, we define  $\bla\dots|\dots\bra^{\Ga_i}$ to be the product of the invariants defined by its connected components.
 
 Because the total curve has genus zero, each component of $\Ga_1$ has at most one
relative tail in common with  each component of $\Ga_2$.
In many cases we will be able to show that $\Ga_1$ is connected, and hence that $\Ga_2$ has $r$ components, one for each relative constraint.

Most of the next result is  known:
 part (i) follows from 
Theorem~1.2 in Hu~\cite{Hu}, while part (ii) is very close to 
\cite[Theorem~6.1]{HLR}.

\begin{lemma}\labell{le:strun} Let $(\TM,\Tom)$ be the one point blow up of $(M,\om)$ with exceptional divisor $E$.\smallskip

\NI{\rm (i) } If $(M,\om)$ is strongly uniruled, then  $(\TM,\Tom)$ is also.\MS

\NI {\rm (ii)}  If $\bla a_1,a_2,pt\bra^{\TM}_\be\ne 0$ for some 
$a_i\in H_*(M) = H_*(\TM\less E)$ and some $\be\in H_2(M)\subset H_2(\TM)$  then $(M,\om)$ is strongly uniruled.
\end{lemma}
\begin{proof} Consider (i).  By hypothesis there is a nonzero invariant $\bla a_1,a_2,pt\bra^M_\be$.  Think of $M$ as the fiber sum of $(\TM,E)$ with $(\PP^n,\PP^{n-1})$ and evaluate this invariant by the decomposition formula, putting all constraints into $\TM\less E$.  
It follows that there is a nonzero relative invariant
$$
\bla a_1,a_2, pt\,|\, E^{j_1},\dots,E^{j_s}\bra^{\Ga_1,\TM,E}_{\Tbe,\un \ell}
$$
with $s\ge0$ that is paired with a nonzero invariant
$\bla \;|E^{n-j_1+1},\dots,E^{n-j_s+1}\bra^{\Ga_2,\PP^n,\PP^{n-1}}_{\ell \la,\un \ell}$.  Note that $\Tbe = \be - \ell\eps$, where $\ell = \sum\ell_i$.  Since there are no absolute constraints
in $\PP^n$, Lemma~\ref{le:0} (i) implies that $s=\ell=0$ and $\Ga_2=\emptyset$.
But then  Lemma~\ref{le:0} (ii) states that
 the above relative invariant equals 
the absolute invariant $\bla a_1,a_2, pt\bra^{\TM}_\be$.
 This proves (i).

Now consider (ii).  
By Lemma~\ref{le:0} (ii)
$\bla a_1,a_2,pt\bra^{\TM}_\be = \bla a_1,a_2,pt\,|\,\bra^{\TM,E}_\be\ne 0$. 
Now use the decomposition formula to evaluate $\bla a_1,a_2,pt\bra^{M}_\be$, again putting all the absolute constraints in  
$\TM\less E$.  Because there are no absolute constraints in 
$(\PP^n,\PP^{n-1})$,
it follows as before that
 there are no terms in this formula with $\Ga_2\ne \emptyset$.
Hence 
$$
\bla a_1,a_2,pt\bra^{M}_\be = \bla a_1,a_2,pt\,|\,\bra^{\TM}_\be\ne 0
$$
as required.\end{proof}

\begin{rmk}\rm  The proof of Proposition~\ref{prop:unicond} given
below can be adapted to show that the statement in 
Lemma~\ref{le:strun} (ii) holds also in the case $\Tbe\cdot E = 1$.
However, it is not clear whether it continues to hold when
$\Tbe\cdot E > 1$.  It is also not known whether  
the strongly uniruled property 
persists under  blow ups along arbitrary submanifolds.
\end{rmk}

We now prove the following version of  
Proposition~\ref{prop:unicond2}.

\begin{prop}\labell{prop:unicond}
If  there is a nonzero invariant 
\begin{equation}\labell{eq:TM}
\bla a_1,\dots,a_p,E^{j_1},\dots, E^{j_q}\bra^{\TM}_{\Tbe},
\end{equation}
with  $ a_i\in H_*(\TM\less E)\cong H_*(M)$, $q\ge 1$ and 
$\Tbe\cdot E>0$ then $(M,\om)$ is uniruled.
\end{prop}
 
We prove this by the method of Hu--Li--Ruan~\cite{HLR}.
Thus we first think of $\TM$ as the fiber (or Gompf) sum of $(\TM,E)$ with
$(X, E^+)$, where $X: = \PP(\Oo(-1)\oplus\C)$ is the projectivized normal bundle to $E=\PP^{n-1}$ and $E^+=\PP(\Oo(-1)\oplus \{0\})\cong \PP^{n-1}$ is the section
with positive normal bundle.  Using this decomposition, we show that the existence of the nonzero absolute invariant (\ref{eq:TM}) implies the
existence of a nontrivial relative invariant  (\ref{eq:Eblow}) for the pair $(\TM,E)$.
Next we identify the blow down $M$ as the fiber sum of $(\TM,E)$ with the pair $(\PP^n,\PP^{n-1})$ and deduce from the nontriviality of
(\ref{eq:Eblow}) the nontriviality of a suitable absolute invariant
for $M$.

In the following we denote by $\eps\in H_2(\TM)$ the class of the line in the exceptional divisor $E$, and write  the class $\Tbe\in H_2(\TM)$ as
$\be - d\eps$, where $d: = \Tbe\cdot E$ and  $\be\in H_2(\TM\less E)=H_2(M)$.  
The homology of the exceptional divisor $E\cong \PP^{n-1}$ is generated by the hyperplane class
in $H_{2n-4}(E)$. As a homology class we identify $E$ with
 the class in $H_{2n-2}(M)$ it represents. Hence the generator of
$H_{2n-4}(E)$ is $E^2$, and $\eps = E^{n-1}$.  
The relative constraints for $(\TM,E)$ have the form $E^j, j=1,\dots,{n}$.  With our conventions (which are different from~\cite{HLR})
the constraint in $E$ dual to 
$E^j$ is $-E^{n-j+1}$, i.e.
\begin{equation}\labell{eq:dual} E^j\cdot_E (-E^{n-j+1}) = pt.
\end{equation}

\begin{lemma}\labell{le:d1} If there is a nonzero absolute invariant 
of the form (\ref{eq:TM}) for a given  $p\ge 0$ and $q\ge 1$
then there is a nonzero (connected) relative invariant  of the form
\begin{equation}\labell{eq:Eblow}
\bla a_1',\dots,a_m'\,|\,E^{i_1}, \dots, E^{i_r}\bra^{\TM,E}_{\be'-\ell\eps, (\ell_1\dots,\ell_r)}, \qquad  i_j\ge 1,\;\;\sum \ell_i = \ell>0,
\end{equation}
where $0\le m\le p$, $\be'\in H_2(M)$ and $a_i\in H_*(M)$. 
 \end{lemma}

\begin{proof} Decompose $\TM$ as the fiber sum  of $\TM$ with the ruled manifold $(X, E^+)$ as above. 
Apply the decomposition formula to
evaluate the nonzero invariant
$$
\bla a_1,\dots,a_p,E^{j_1},\dots, E^{j_q}\bra^{\TM}_{\Tbe}
$$
putting all the $a_1,\dots,a_p$ insertions into $\TM\less E$ and the insertions $E^{j_1},\dots, E^{j_q}$ into $X\less E^+$.  Because $\Tbe\notin H_2(E)$, each term in the decomposition formula must correspond to a splitting $\Tbe = \Tbe'+\al$ where $0\ne\Tbe'\in H_2(\TM)$.  Hence  there is a nonzero relative invariant for $(\TM,E)$  in some class $\Tbe'\ne 0$ that goes through all the constraints
$a_1,\dots,a_p$.
 It counts 
 curves modelled on the possibly disconnected graph $\Ga_1$ and hence is a product of connected invariants, each of which has the  form 
(\ref{eq:Eblow}) for some subset of $a_1,\dots,a_p$. 
Note that each such  connected invariant has nonempty intersection with $E$ because the initial $\Tbe$-curve in $\TM$ is connected
and  $q\ge 1$. 
\end{proof}

\begin{lemma}\labell{le:d2} 
 If there is a nonzero relative invariant for $(\TM,E)$ of the form
(\ref{eq:Eblow}) for some $m\ge 0$ and $r\ge 1$ then there is a nonzero absolute invariant  on $M$
$$
\bla a_1,\dots, a_t, \tau_{k_1}pt,\dots,\tau_{k_s} pt\bra^{M}_{\be} 
$$
for some $\be\ne0$, $t\le m$, $1\le s\le r$, $ k_j\ge 0$ and $a_i\in H_*(M).$  
\end{lemma}
\begin{proof}
Choose a class $\be$ of minimal energy $\om(\be)$ such that there  is a nonzero connected relative invariant in some class
$\Tbe: = \be-\ell \eps, \ell>0,$ of the form (\ref{eq:Eblow}) with $t\le m$ absolute constraints.  
Denote by
 $s$ the smallest $r>0$ such that there is a nonzero 
invariant (\ref{eq:Eblow}) with $m=t$,  this class $\be$, and $r$ relative constraints.
Denote 
the corresponding relative constraints by 
$E^{n-j_1+1}, \dots, E^{n-j_s+1}$ (cf. equation~(\ref{eq:dual})) and the  multiplicities by $\un \ell= (\ell_1,\dots,\ell_{s})$. 
Note that $\sum \ell_j=\ell>0$.  In particular $s>0$.  

  Decompose $M$ into the fiber sum of $\TM$ with $Y=\PP^n$ by identifying $E\subset \TM$ with the hyperplane $E^+=\PP^{n-1}$ in $Y$.  
  Apply the decomposition formula to
evaluate 
\begin{equation}\labell{eq:Eblow2}
\bla a_1,\dots,a_t, \tau_{k_1}pt,\dots,
\tau_{k_{s}} pt\bra^{M}_{\be},\quad k_i: = n\ell_i - j_i,
\end{equation}
putting all the point insertions into $Y$ and the others into $\TM\less E$.  We claim that (\ref{eq:Eblow2}) is nonzero.

Note first that by Lemma~\ref{le:0}(iii) there is a 
nonzero term $T$ in this decomposition formula given by taking the product of the nonzero relative invariant (\ref{eq:Eblow}) with $s$ terms of the
form $\bla \tau_{k_i}pt \,|\, E^{j_i}\bra^{\PP^n,\PP^{n-1}}_{\ell_i \la, (\ell_i)}$. We need to check that all other terms in this  formula vanish.

Consider an arbitrary nonzero  term $T'$ in this formula.  It is
a product of a
 relative invariant for $(\TM,E)$ modelled on a graph $\Ga_1$ and
in some class $\be'-d\eps$
 with a relative invariant for $(\PP^n, \PP^{n-1})$ in 
 class $d \la$.  Since the classes $\be'-d\eps$ and $d\la$ combine to give $\be'$, we must have $\be' = \be$.  The minimality of $\om(\be)$ implies that 
 $\Ga_1$ is connected, since otherwise each of its components is in a class $\be_i-d_i\eps$ with $0<\om(\be_i)< \om(\be)$.

Now look at the other side.
  By Lemma~\ref{le:0} (i) each connected 
 relative invariant in
 $(\PP^n, \PP^{n-1})$ must go through some  absolute constraint. 
 Hence $\Ga_2$ has at most $s$ components.  But it cannot have fewer than $s$ components, because if it did  $\Ga_1$ would 
 have fewer than $s$ relative constraints (since the genus zero 
 requirement means that $\Ga_1$  meets each component of $\Ga_2$ at most once), which contradicts the minimality of $s$.  Therefore there are $s$ components of $\Ga_2$ and each goes through precisely one absolute constraint.  Thus 
 each component is
a nonzero $2$-point invariant of the form
$\bla \tau_{k_i}pt \,|\, E^{j}\bra^{\PP^n,\PP^{n-1}}_{d f, (d)}$
for some $i, d$ and $1\le j \le n$.  But $k_i$ has a 
unique decomposition of the form
$n\ell_i - j_i$ where $1\le j_i \le n$.  Hence, by Lemma~\ref{le:0} (iii)
 the relative constraints 
  in $\Ga_2$ are precisely the $E^{j_i}, i = 1\dots,s,$ with intersection multiplicities $\un\ell$.
  Thus $T'=T$. \end{proof}

\NI {\bf Proof of Proposition~\ref{prop:unicond}.} 
By Lemma~\ref{le:d2} it suffices to show that 
$M$ is uniruled iff an invariant of the form
$\bla \tau_{i_1}a_1, \dots, \tau_{i_{k-1}}a_{k-1}, \tau_{i_k}pt\bra^M_{\be}$ is nonzero.
 Hu--Li--Ruan~\cite[Theorem~4.9]{HLR} prove this by an inductive argument (similar to the proof of Proposition~\ref{prop:usu} below)
that is based on the
identity (\ref{eq:LPi}).
\QED

\subsection{Blowing down section invariants.}
\labell{ss:eps}

 It remains to prove Proposition~\ref{prop:eps}.
 We begin with a general remark about
 the  invariants of pairs of spaces $(X,D)$ that are fibered over $\PP^1$ with fiber $(F, D_F)$.   It concerns
 invariants in the class $\be\in H_2(F)$ of a fiber.  Denote by $\io:H_*(F)\to H_*(X)$ the map induced by inclusion.

\begin{lemma}\labell{le:fib} Let $\pi: X\to \PP^1$ be a Hamiltonian fibration with fiber $F$ that induces a fibration $F_D\to D\to \PP^1$
on the divisor $D$. If  $a_1:= \io(a_1)$ is a fiber class  and $\be \in H_2(F)$ then
$$
\bla \io(a_1),\dots\,|\, b_1,\dots\bra^{X,D}_{\be, \un d} = \bla a_1, a_2\cap F,\dots\,|\, b_1\cap D_F,\dots\bra^{F,D_F}_{\be, \un d}.
$$
Similarly, if
$b_1: = \io(b_1)$ is a fiber class then
$$
\bla a_1,\dots\,|\, \io(b_1),\dots\bra^{X,D}_{\be, \un d} = \bla
 a_1\cap F,\dots\,|\, b_1, b_2\cap D_F,\dots\bra^{F,D_F}_{\be, \un d}.
$$
In particular, any invariant with two fiber insertions must vanish.
Corresponding results hold for absolute invariants.
\end{lemma}
\begin{proof} This holds because one can define $\oMm\,\!^{\,[vir]}_\be$
in such a way that each of its elements represents a curve lying in 
some fiber of $\pi$.  Therefore, if we represent the class $a_1$ by a cycle lying in the fiber $F_z: = \pi^{-1}(z)$ all the curves
that contribute to the invariant lie entirely in this fiber and hence must intersect the other cycles in representatives for $a_i\cap F, b_j\cap F$.
The details of the proof  in the absolute case are spelled out
 in \cite[Prop~1.2(ii)]{Mcq}.  The relative case is similar.
\end{proof}

We saw earlier that $\TP$ is the blow up of $P$ along a section $s_{\max}$ in $F_{\max}\times \PP^1$. Similarly,
$\TP'$ is the blow up of $P'$ along the corresponding section $ s_{\min}'$ in 
$F_{\min}'\times \PP^1$. Note that $s'_{\min}$ has normal bundle $\Oo(1) \oplus\C^{n-1}$,   where   $\Oo(m)\to \PP^1$ 
denotes the holomorphic line bundle of Chern class $m$.   
We denote by $D$ the exceptional divisor in $\TP'$.
Thus it is a $\PP^{n-1}$-bundle 
$$
\pi_D: D: = \PP(\Oo(1) \oplus\C^{n-1})\to \PP^1.
$$
Consider the section $s_D: = \PP( \{0\}\oplus \dots\oplus\{0\}\oplus \C)$
of $D\to \PP^1$.  It lies with trivial normal bundle in the
product divisor $V: =\PP(\{0\}\oplus\C^{n-1})\cong 
s_D\times \PP^{n-2}$ and so $c_1(s_D) = 3$.
The exceptional divisor $E$ of $\TM$ can be identified with
$\pi_D^{-1}(pt)$.  The line $\eps$ in $E$ intersects $V$ once.
Therefore classes $s_D-m\eps$ with $m>1$ have no holomorphic representatives in $D$ since they have negative intersection with $V$ and cannot be represented in $V$ itself.  On the other hand
the section class $s_Z: = s_D-\eps$ has the unique representative   
$  \PP(\Oo(1) \oplus \{0\}\oplus \dots\oplus\{0\})$.

The proof of Proposition~\ref{prop:eps} has two steps: we need to show that the absolute invariant in $\TP'$ equals a suitable relative invariant for $(\TP',D)$, and then look at the blow down correspondence.  The first of these steps turns out to be the hardest,
 requiring a detailed study of the 
 the invariants of 
 the ruled manifold
 $(Y,D^+)$.  Here $Y$ is the projectivization $\PP(L\oplus \C)$ of the normal bundle to $D\subset \TP'$ and $D^+: = \PP(L\oplus \{0\})$ as above; see Fig.~\ref{fig:1}.  
Thus $(Y,D^+)$ fibers over $\PP^1$ with fiber $(X,E^+)$
 equal to the pair consisting of the $1$-point blow up of $\PP^n$ and hyperplane class $E^+$, and there is a
 commutative diagram
$$
\begin{array}{ccccc}
\PP^1&\to & (X,E^+) &\stackrel{\rho_X}\to& E\\
\downarrow&&\downarrow&&\downarrow\\
\PP^1&\to& (Y,D^+)&\stackrel{\rho_Y}\to& D\\
&&\!\!\pi_Y\downarrow\;\;\;&&\!\!\!\pi_D\downarrow\;\;\;\\
&& \:\PP^1&=&\;\;\PP^1.\end{array}
$$
We denote by $D_0$ the divisor
$\PP(\{0\}\oplus \C)$ in $Y$ that is disjoint from $D^+$, and by
$s_{D_0}$ the section in $D_0$ corresponding to $s_D.$  
Since $D_0$ has the same normal bundle as $D\subset \TP'$,
$c_1(s_{D_0}) = 3$. Also $D_0\cap X$ can be identified with the exceptional divisor $E$ in $\TM$.  Thus $c_1(\la_X) = n-1$ where $\la_X$ is the class of a 
line in $X\cap D_0$. Note that $H_2(Y)$ is generated by $s_{D_0}, \la_X$ and $f$.  Moreover $\rho_Y(s_{D_0} - \la_X) = s_Z.$ 
Thus we write $s_{Z_0}: = s_{D_0} - \la_X$.
\begin{figure}[htbp] 
   \centering
   \includegraphics[width=4in]{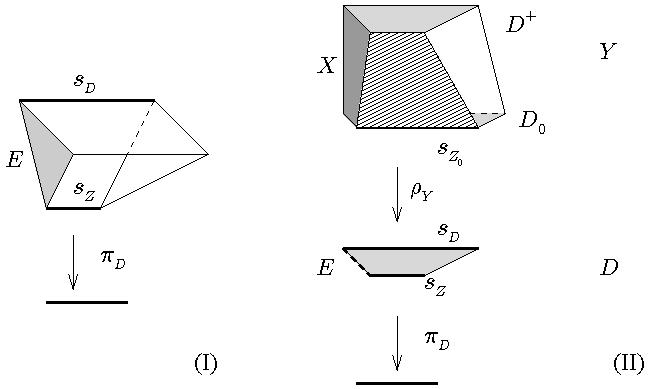} 
   \caption{Diagram (I) is a $3D$ picture of the moment polytope of the toric manifold $D$, while (II) is a $3D$ picture of that for $Y$ with $D$ reduced to $2$-dimensions}
   \label{fig:1}
\end{figure}

 The next lemma contains
 some preliminary calculations.
Here, and subsequently, we choose the relative constraints 
$b_i^*$ from the elements of the following self-dual 
basis for $H_*(D)$:
$$
D^i\in H_{2(n+1-i)}(D),\; 1\le i\le n,\quad
E^j\in H_{2(n-j)}(D),\; 1\le j\le n.
$$
As always $E^j$ denotes the image in $D$ of 
the class in $\TM$ represented by the $j$-fold intersection 
of $E$ with itself in $\TM$.
Thus the $E^j$ are fiber constraints for the projection $\pi_D: D\to \PP^1$, while the $D^i$ are nonfiber constraints.  Also $D^i\cdot E^j=0$ unless $i+j=n+1$.


\begin{lemma}\labell{le:YY} (i)  If $m\le -2$, all invariants of $(Y,D^+)$
in the classes $s_{D_0} +m\la_X + df$ vanish.
They also vanish in the classes $m\la_X + df$ with $m\le -1$.
\SSS

\NI (ii) Let $j = 1$ or $2$. Then $\bla E^j|b_1^*,\dots,b_r^*\bra^{Y,D^+}_{s_{D_0} + df + m\la_X, \un d} = 0$
for all $d$ such that  $0\le d\le m+1$.\SSS

\NI (iii) Let $d>0$ and $1\le j\le 2$.  If $\bla E^j|b_1^*,\dots,b_r^*\bra^{Y,D^+}_{df + m\la_X, \un d} \ne 0$
then $b_i^*$ is a nonfiber constraint for all $i$. Moreover if 
$n\ge 3$ and $r=1$ then $m= 0$ and $d=1$. The same holds if $n=2$ and $j=r=1$.
\SSS

\NI (iv) 
If $\bla \;|b_1^*,\dots, b_r^*\bra^{Y,D^+}_{df + m\la_X, (d)} \ne 0$
for some $m\ge 0$ and $1\le r\le 2$ then $m< d$.
\end{lemma}

\begin{proof} (i)  Since $\rho_Y: Y\to D$ is holomorphic 
every class with nonzero GW invariants in $Y$ must map to a class with
a holomorphic representative in $D$.  Hence their image must have nonnegative intersection with the divisor $V$ containing $s_D$.
Since $s_D\cdot V = 1$ and $\eps\cdot V = -1$ the result follows.\SSS  

Now consider (ii).
A dimensional calculation shows that when $m+1\ge d$ 
the invariant (ii) can be nonzero only if $m=-1$, $j=2$ and $d=r=0$.  Hence it remains to consider the invariant
$\bla E^2|\;\bra^{Y,D^+}_{s_{Z_0}}$, where $s_{Z_0}: = s_{D_0}-\la_X$. 
Although this could 
be nonzero as far as dimensions are concerned, the geometry shows that it must be zero. To see this, 
consider the holomorphic fibration $\pi_Y:Y\to \PP^1$.  
 Since the normal bundle to $s_Z$ in $D$ is a sum of copies of 
 $\Oo(-1)$, there is a unique curve in $D$ in class $s_Z$.  
 Since any holomorphic curve in $Y$ in class $s_{Z_0}$  projects to an $s_Z$ curve in $D$, the $s_{Z_0}$ curves in $Y$ lie in the surface
 $\rho_Y^{-1}(s_Z)$; cf. Figure~\ref{fig:1} where this surface is 
 crosshatched.  This surface has $2$-dimensions, while the constraint
 $E^2$ has codimension $3$ in $Y$.  (Remember that $E$ is the exceptional divisor in $\TM$ and $E^2$ denotes its intersection in $\TM$.) Hence there is a cycle in $Y$ representing $E^2$ that does not meet any $s_{Z_0}$ curves. This proves (ii).
 
 \SSS
 
  In case (iii) we are counting curves in a fiber class
 of the holomorphic projection $\pi_Y:Y\to \PP^1$. 
 Since $E^j$ is a fiber constraint, the first claim is an immediate consequence of Lemma~\ref{le:fib}.  To prove the second, note that by Lemma~\ref{le:fib} the invariant reduces to one in $(X,E^+)$.  A count of dimensions shows that $m=0$.  But then $f$ is a fiber class for the projection $\rho_X$ so that the invariant is nonzero only if
 $\rho_X(b)\cap \rho_X(E^j)\ne 0$.  A further dimension count now implies that $d=1$ (or one can use the fact that one is counting $d$-fold covered spheres in $\PP^1$ with $2$ homological constraints).  The proof of the third claim is similar.
\SSS

Consider (iv).   Since the fibers $X$  of the bundle $Y\to \PP^1$
can all be identified (as K\"ahler manifolds) with the $1$-point blow up of $\PP^n$ (with $X\cap D^+$ the hyperplane class),  
this invariant can be nonzero only if one of the relative constraints, say $b_1^*$  is a fiber constraint while the others are not.  Therefore  $\de_b\ge r-1$.  Now count dimensions.  
\end{proof}

\begin{lemma}\labell{le:eps1} Suppose that 
$(M,\om)$ is not uniruled and that  $n\ge 3$.
Then  $$
\bla\;  |E^2\bra^{\TP',D}_{\si-\eps}
= \bla E^2\bra^{\TP'}_{\si-\eps}.
$$
\end{lemma}
\begin{proof}
As in the proof of Lemma~\ref{le:d1}
decompose $\TP'$ as the connected sum of $(\TP',D)$ with the ruled manifold $(Y,D^+)$ considered above.
Consider a typical term  in the decomposition formula
for $\bla E^2\bra^{\TP'}_{\si-\eps}$ where we put the constraint
$E^2$ in $D_0\subset Y$:
$$
\bla \;  |b_1,\dots,b_r\bra^{\Ga_1,\TP',D}_{\al_1-\ell\eps,\un \ell}\;
\bla E^2 |b_1^*,\dots,b_r^*\bra^{\Ga_2,Y,D^+}_{\al_2+\ell f,\un \ell}
$$
where $\al_1\in H_2(P')$ and $\al_2\in H_2(D_0)$.
(Note that the intersections on each side of $D$ match because $\al_1\cdot D = 0$ for all $\al_i\in H_2(P')$.)
There is one term in this formula with $\al_1=\si, \ell=r=1$ and $b_1=E^2$.
The first factor is then $\bla\;  |E^2\,\bra^{\TP',D}_{\si-\eps}$
while the second is 
$\bla E^2|D^{n-1}\bra^{Y,D^+}_f = 1.$
Hence we must see that all other terms vanish. 

The argument has several steps. Throughout the following discussion
the words {\it fiber/section class} and {\it fiber/nonfiber constraint}
apply to the fibration $\pi$ over $\PP^1$. Moreover a component of 
$\Ga_1$ or $\Ga_2$ is called a {\it fiber component} if it represents a class in the homology $H_2(\TM)$ of the fiber.
\SSS

\NI (a) {\it  No component of $\Ga_1$ can lie in a fiber class $\Tbe$ and go through a fiber constraint $b$.}   For if there were such a component, 
Lemma~\ref{le:fib} would imply that the corresponding invariant 
$\bla\;| b, b_2,\dots\bra^{\TP',D}_{\Tbe,\un k}$ equals the fiber invariant
 $\bla\;| b, b_2\cap \TM,\dots\bra^{\TM,E}_{\Tbe,\un k}$,
 and so $(M,\om)$ would be uniruled by Lemma~\ref{le:d2}.\SSS

\NI (b) {\it Every fiber component of $\Ga_2$ has $r\le 2$; moreover a fiber component through $E^j$ has $r=1$.}  If a component of $\Ga_2$ 
has two nonfiber relative 
constraints, say $b_1^*, b_2^*$, then the genus zero restriction implies that
at least one of the dual fiber constraints $b_1, b_2$ must lie in a fiber component of $\Ga_1$ (since $\Ga_1$ has at most one section component.)  But this contradicts Step (a). Thus the second claim holds by the first part of Lemma~\ref{le:YY}(iii), while the first
claim holds by Lemma~\ref{le:fib} 
which says that
only one of the relative constraints for each component can
be a fiber class.    \SSS

\NI (c) {\it If $\Ga_2$ is a section class, its section component must go through the absolute constraint.}  For otherwise $\Ga_2$ contains an  invariant of 
the type considered in Lemma~\ref{le:YY} (iii) whose relative constraint is dual to a fiber constraint. Hence $\Ga_1$ would have to contain a fiber component with a fiber constraint, contradicting (a).\SSS

\NI (d) {\it $\Ga_2$ cannot be a section class}.  If $\Ga_2$ is a section class then by (c) the fiber components  in $\Ga_2$ have no absolute constraints.  Hence by Lemma~\ref{le:YY}(i) and (iv) (which applies by (b))
the fiber classes 
have the form $d_i f + m_i\la_X$ with $0\le m_i < d_i$. Suppose the section class is
 $s_{D_0} + d_0f + m_0\la_X$.  Then these classes in $\Ga_2$ combine with $\al_1-d\eps$ to make $\si-\eps$ where 
$\si\in H_2(P')$.  But $\la_X$ projects to the class $\eps$ in $D$.
Hence 
$$
-d + \sum_{i\ge 0} m_i = \sum_{i\ge 0} (m_i  - d_i) = -1.
$$
We saw above that $m_i  - d_i\le -1$ for all $i>0$. Therefore either $\Ga_2$ is connected and  $m_0-d_0= -1$ or $m_0-d_0\ge0$.  Neither of these cases occur by 
Lemma~\ref{le:YY}(ii).\SSS


\NI (e)  {\it Completion of the proof.}\,
As in step (d) $d= \sum d_i$ and $\sum_{i\ge 0} (m_i  - d_i) = -1$.
The component of $\Ga_2$ through $E^2$ has $r=1$ by (b) and hence
$m_0=0$ and $d_0=1$ by  Lemma~\ref{le:YY}(iii).  Moreover all other components of $\Ga_2$ have $m_i\ge 0$ by Lemma~\ref{le:YY}(i) and hence $m_i<d_i$ by Lemma~\ref{le:YY}(iv).  Therefore $\Ga_2$ is connected and lies in class $f$.  Hence $\Ga_1$ is also connected 
 and lies in class $\si-\eps$.  Therefore this is the term already considered.
\end{proof}

We now relate the relative invariant in $\TP'$ to an absolute invariant in $P'$.
To this end, consider the $\PP^{n}$-bundle 
$W: =\PP(\Oo(1) \oplus\C^{n-1}\oplus \C)\to \PP^1$. It contains a copy 
$D^+: = \PP(\Oo(1) \oplus\C^{n-1}\oplus \{0\})$ of $D$ with normal bundle $L_D^*$ (where as usual $L_D$ is the normal bundle of $D$), and a disjoint section
$s_W: = \PP(\{0\} \oplus\{0\}\oplus \C)$. This can be identified with $s_D$ and so $c_1(s_W) = 3$. We also denote by $\la$ the line in the fiber of $W\to \PP^1$.  

\begin{lemma}\labell{le:W} Let $n\ge 2$.  Then:\SSS

\NI (i) 
 $\bla \tau_1 pt|b_1^*,\dots,b_r^*\bra^{W,D^+}_{s_W+d\la,\un d} = 0$
 for all $b_i^*\in H_*(D^+)$.\SSS
 
 \NI (ii) $\bla \tau_1 pt|b_1^*,\dots,b_r^*\bra^{W,D^+}_{d\la,\un d} \ne 0$ only if $d=r=1$ and $b_1^*$ is a multiple of $(D^+)^{n-1}$.
 
  \NI (iii) If $d>0$, $\bla \;|b_1^*,\dots,b_r^*\bra^{W,D^+}
  _{d\la,\un d} = 0$ for all $b_i^*$.
 \end{lemma}
 \begin{proof} (i)  Since $c_1(s_W) = 3$ and $c_1(\la) = n+1$ this 
 invariant is nonzero only if
 $$
 n+1 + 3 + d(n+1) + r+1-3 + \de_b -1 -(d-r) = (r+1)(n+1),
 $$
 which reduces to $(d-r)n + \de_b + r=0$.  Hence it
 cannot hold when $d\ge r\ge 1$.  
 
 In case (ii) the dimensional condition is the same except that LHS is reduced by $3=c_1(s_W)$. If
 $d> r$ then we need $n=2, r=1, d=2$ and $\de_b=0$.  But then 
 the invariant has two point constraints
 and hence vanishes by Lemma~\ref{le:fib}: it counts curves in a fiber class through  points that could lie in different fibers. If $d=r$, 
 we need $\de_b+r=3$.  If $r>1$ then at least one of the $b_i^*$ is a point constraint, and the invariant vanishes as before.
 Hence we must have $r=1$, $\de_b=2$.  Since $b$ cannot be a fiber class, it reduces to 
 $\bla \tau_1 pt|b\cap F\bra^{\PP^n,\PP^{n-1}}$ which
 is nonzero when $b= (D^+)^{n-1}$ by Lemma~\ref{le:0}(iii). 
 
Claim (iii) holds by a dimension count (or by
Lemma~\ref{le:0}(i)).
 \end{proof}
 
\begin{cor}\labell{cor:eps2} Suppose that $(M,\om)$ is not uniruled and that $n\ge 2$. Then there is $c\ne0$ such that
$\bla \tau_{1}pt\bra^{P'}_\si =
c\bla \; |E^2\bra^{\TP',D}_{\si-\eps}$.
\end{cor}
\begin{proof}
Identify $P'$ with the fiber sum of $(\TP',D)$ with
$(W,D^+)$ and calculate $\bla \tau_{1} pt\bra^{P'}_\si$ using the decomposition formula, putting the point into $W\less D'$.
A typical  term in the decomposition 
formula~(\ref{eq:DF})  has the form
$$
\bla  \; |b_1,\dots,b_r\bra^{\Ga_1,\TP',D}_{\al_1-d\eps,\un d}\;
\bla \tau_{1}pt |b_1^*,\dots,b_r^*\bra^{\Ga_2,W,D^+}_{\al_2+d\la,\un d}
$$
where $\al_2=0$ or $s_W$, and $\la$ is the class of a line in the fiber of $W$. 

We saw in Lemma~\ref{le:W} that in a nonzero term $\al_2=0$. Further each nonzero component of $\Ga_2$  must have an absolute constraint
and $d=r=1$.  Hence
there is only one such component and it has
$b_1^* = (D^+)^{n-1}$.   Hence $b_1=E^2$ and  $\Ga_1$ must consist of the single component 
$\bla  \; |E^2\bra^{\TP',D}_{\si-\eps,(1)}$. Hence there is only one term in the decomposition formula. The result follows.\end{proof}

\NI {\bf Proof of Proposition~\ref{prop:eps}(i).}

This follows immediately from Corollary~\ref{cor:eps2} and Lemma~\ref{le:eps1}.\QED

We now turn to the proof Proposition~\ref{prop:eps}(ii) 
which concerns the case $n=2$.

\begin{lemma}\labell{le:eps3} Let $n=2$.  If
  $\bla E\bra^{\TP'}_{\si-2\eps} \ne 0$ then
 $  \bla \; |E,D\bra^{\TP',D}_{ \si-2\eps,(1,1)} \ne 0.$  
 \end{lemma}
   \begin{proof}
As before, we calculate the nonzero invariant
   $\bla E\bra^{\TP'}_{\si-2\eps}$ by considering
   $\TP'$ as a connected sum of $(\TP',D)$ with $(Y,D^+)$, putting $E$ into $Y$.  Steps (a), (b) and (c) of the proof of
    Lemma~\ref{le:eps1} go through as before.  So consider  (d) and suppose that $\Ga_2$ is a section class.  The homology class count is now
   $$
\sum_{i\ge 0} (m_i-d_i) = -2. 
  $$
   By Lemma~\ref{le:YY} (i), $m_0\ge -1$ and $m_i\ge 0$.  Therefore 
   by  Lemma~\ref{le:YY} (ii) $m_0 - d_0 < -1$ and by  
   Lemma~\ref{le:YY} (iv) $m_i-d_i\le -1$ for $i>0$. Therefore
   $\Ga_2$ must be connected and $m_0 - d_0 =-2$.  Since
   $d_0\ge 0$ this gives $m_0\le -2$, which is impossible.  
  
   Since $m_i-d_i\le -1$ for fiber constraints, this argument  also shows that $\Ga_2$ has at most two fiber components.
   One of these goes through the $E$ constraint and has $m=0,d=1$
   by  Lemma~\ref{le:YY} (iii).  Hence there is exactly one other 
  that has $0\le m'= d'-1$ and $r\le 2$.   Since this is a fiber invariant it reduces to an invariant in $(X,E^+)$ of the form
  $$
  \bla\;|a_1,\dots,a_r\bra^{X,E^+}_{d'f + (d'-1)\la_X, \un d'}.
  $$
Since $n=2$, this is possible only if $d'=1$. 
 Hence both components of $\Ga_2$ lie in the class $f$, and the $\Ga_2$ factor 
must be
$$
 \bla E |D^2\bra^{Y,D^+}_{f} \; \bla \;|pt\bra^{Y,D^+}_{f}.
 $$
Hence
 $
  \bla \; |E, D\bra^{\TP',D}_{ \si-2\eps,(1,1)} \ne 0.
  $
\end{proof}

  \NI {\bf Proof of Proposition~\ref{prop:eps}(ii).}\SSS
  
  Blow down $D$  by summing with $(W,D^+)$ as in Corollary~\ref{cor:eps2}. Consider the resulting 
  decomposition formula for
  $\bla pt,s_W\bra^{P'}_\si$ where the absolute constraints
  are put into $W$. \SSS

  \NI (a)  {\it There is no term in this formula with $\Ga_1=\emptyset$.} 
   In such a term $\si$ would be a section class in $W$ 
   with $c_1^W(\si) = 3$. Since the only section classes are $s_W+m\la$ and $c_1^W(\la) = 3$, we must have $\si=s_W$. 
   But  the only holomorphic representatives of $s_W$ that meet the product divisor $Z = \PP(\{0\}\oplus \C^2)\cong s_W\times \PP^1$ are parallel copies of $s_W$ lying entirely in $Z$.  Hence one can place the constraints $pt, s_W$ in $Z$ in such a way that no holomorphic  $s_W$-curve meets them both.
   \SSS
   
   \NI (b)  {\it $\Ga_2$ cannot be a section class.} 
   For if so, by  (a), there would have to be a nonzero invariant
   $$
   \bla a_1,\dots,a_m| b_1^*,\dots,b_r^*\bra^{W,D^+}_{s_W+d\la,\un d},\quad d>0,
   $$
   where $0\le m\le 2$ and the $a_i$ form a subset of $\{pt, s_W\}$.
   This is impossible for dimensional reasons.
   \SSS
   
   \NI (c) {\it $\Ga_2$ has at most $2$ components. 
    Each  has $d=r=1$}.
By (b) each component in $\Ga_2$ lies in some class $df$. 
Even if one puts all the constraints into one component, a dimension count shows that
the invariant
$\bla pt, s_W|b_1,\dots,b_r\bra^{W,D^+}_{d\la,\un d}$ 
is nonzero only if $d=1$.  It follows that $d=1$ in all cases.
Thus we are counting lines, and so
each component must have two constraints and hence since $r=1$
each must have an absolute constraint.\SSS

Therefore the only possibilities for $\Ga_2$ are:
one component of type
 $$
\bla pt, s_W|D\bra^{W,D^+}_{\la,(1)}= 
\bla pt, pt|E^+\bra^{F,E^+}_{\la,(1)},
$$
or the two terms in the product
$$
    \bla pt |s_D\;\bra^{W,D^+}_{\la,(1)} \; \bla s_W |pt\;\bra^{W,D^+}_{\la,(1)}.
$$ 
The corresponding $\Ga_1$ terms are 
$$
\bla \;|pt\bra^{\TP',D}_{\si-\eps},\quad\mbox{ and }\;\; 
\bla \;|D,E\bra^{\TP',D}_{\si-2\eps}.
$$
We saw in Lemma~\ref{le:eps3} that the second of these terms is nonzero. Therefore, if the first term vanishes, 
  $\bla pt,s_W\bra^{P'}_\si\ne 0$.  However, if the first term does not vanish, we may apply Corollary~\ref{cor:eps2} to conclude
  that $\bla \tau_1 pt\bra^{P'}_\si\ne 0$.
  This completes the proof.\QED

\subsection{Identities for descendent classes}\labell{ss:LP}
 
 To complete the proof of Theorem~\ref{thm:main} we must establish 
  Lemma~\ref{le:horpt2}.   Its proof is based on  
some identities for genus zero Gromov--Witten invariants
that we now explain.
We shall denote by $\psi_i$ the first Chern class of the cotangent bundle to the domain of a stable map at the $i$th marked point, and by $\oWw_\be$
the configuration space of all (not necessarily holomorphic) stable maps 
in class $\be$.
If $i\ne j$ consider the subset $D_{i,\be_1\,|\,j,\be_2}$  of $\oWw_\be$ consisting of all stable maps with
at least two components, one in class $\be_1$ containing $z_i$ and the other in class $\be_2: = \be-\be_1$ containing  $z_j$.
Further, if $i,j,k$ are all distinct consider the subset $D_{i|jk}$ 
of  $\oWw_\be$ 
 consisting of all stable maps with
at least two components, one in some class $\be_1$ containing $z_i$ and the other in class $\be-\be_1$ containing  $z_j, z_k$. 
 The 
virtual moduli cycle $\oMm\,\!^{\,[vir]}_\be$ 
has a natural map to $\oWw_\be$
that may be chosen to be transverse to the above subsets.  These therefore pullback to real codimension $2$ sub(branched) manifolds
in $\oMm\,\!^{\,[vir]}_\be$ which we denote by the same names. 

In \cite[Thm~1]{LP}, Lee--Pandharipande prove the following identities in
the algebraic case:
 \begin{eqnarray}\labell{eq:LPi}
\psi_i &=& D_{i|jk}\\\labell{eq:LPii}
ev_i^*(L)& = &ev_j^*(L) + (\be\cdot L)\,\psi_j -\sum_{\be_1+\be_2=\be}
(\be_1\cdot L)\, D_{i,\be_1\,|\,j,\be_2},
\end{eqnarray}
where the two sides are considered as elements of an appropriate Picard group and $L\in H_{2n-2}(M;\Z)\cong H^2(M)$ is any divisor. We might try to think of  
these equations as identities in
the second cohomology group of $\oMm\,\!^{\,[vir]}_\be$.
However, $\oMm\,\!^{\,[vir]}_\be$ is not 
really a space in its own right
since it is only well defined up to certain kinds of cobordism.
  Hence it does not have a well defined $H^2$.  
Therefore  we shall interpret these equations as 
 identities for 
Gromov--Witten invariants.  For example, (\ref{eq:LPi}) 
states that for all $k\ge 3$, all $i_1>0$, $i_j\ge 0$ and all classes $\be\in H_2(M), a_i\in H_*(M)$,
$$
\bla \tau_{i_1}a_1,\dots, \tau_{i_k}a_k\bra^M_{\be} =
\sum_{j, S, \be_1+\be_2=\be} 
\bla \tau_{i_1-1}a_1,\dots,\xi_j \bra^M_{\be_1}\;\bla \xi_j^*, \tau_{i_2}a_2, \tau_{i_3}a_3,\dots\bra^M_{\be_2}.
$$
Here the sum is over the elements of a basis $\xi_j$ of $H_*(M)$ (with dual basis $\xi_j^*$), all decompositions $\be_1+\be_2=\be$ of $\be$, and all distributions $S$ of the constraints $\tau_{i_4}a_4,\dots, 
\tau_{i_k}a_k$ over the two factors, subject only to the restriction that if $\be_1=0$ then 
the first factor must include at least one of the
$\tau_{i_\ell}a_\ell, \ell\ge 4,$ so that it is stable.  The second identity has a similar interpretation; cf. Lemma~\ref{le:LP}.

Identity (\ref{eq:LPi}) is proved for the space of genus zero stable curves
$\oMm_{0,k}$
by Getzler  at the beginning of \S4 in \cite{Ge}.
  It holds for  stable maps because
the descendent class
$\psi_i$ differs from the pullback of the corresponding class on 
$\oMm_{0,k}$ precisely 
by the boundary class consisting of the elements in
$\oMm\,\!^{\,[vir]}_\be$ such that the 
component containing the $i$th marked point is unstable.
Note that the proof of Proposition~\ref{prop:unicond} and hence of Theorem~\ref{thm:main} requires the full strength of 
this identity.

Lee--Pandharipande's proof of 
(\ref{eq:LPii}) in \cite[\S1.1]{LP}  adapts readily to the symplectic case provided that one has a good framework to work in: one needs a model for the virtual cycle $\oMm\,\!^{\,[vir]}_\be$ in which  the elements in $H_2(\oMm\,\!^{\,[vir]}_\be)$ have \lq\lq nice" representatives.
One could work with ad hoc methods as in \cite{Mcq} and interpret the idea of transversality using the normal cones to the strata of the 
moduli space provided by the gluing data, but it is cleaner to work in the category of polyfolds newly introduced by Hofer--Wysocki--Zehnder
since this provides the moduli spaces with a smooth structure.  
Note that because one needs to use multivalued perturbations
 the moduli spaces are in general branched.  (Here one can work with various essentially equivalent definitions, see \cite[Def~3.2]{Mcbr}
 and \cite[Def~1.3]{HWZ3}.)  
  
To prove Theorem~\ref{thm:main} we only need the very special case of this identity relevant to Lemma~\ref{le:horpt2}.
We will sketch the proof in this case to give the general idea.
Note that here $i=1,j=2$ and the term 
$ev_2^*(L)$ does not contribute since $L\cap pt = 0$.
When we prove Lemma~\ref{le:horpt2} we will explain how 
to construct the moduli space $B$ in this context; 
since we are working in a fibered $6$ dimensional manifold, 
it is not hard to obtain it by ad hoc methods. 

\begin{lemma}\labell{le:LP}  Given a class $\si\in H_2(M)$
and a divisor 
 $H\in H_{2n-2}(M)$, denote by $B: = \oMm^{\,[vir]}_{\si,2}(pt)$
the virtual moduli space of stable maps  
in class $\si$ with two marked points, one through
 a fixed point $x_0\in M$ and the other through a cycle representing $H$.  Suppose that $B$  can be constructed as a smooth branched manifold of real dimension $2$ that is transverse to the strata in $\oWw$.  Then, for any divisor $L\in H_{2n-2}(M)$
\begin{eqnarray*}
\bla LH, pt\bra_\si^M &=& (\si\cdot L)\,\bla H,\tau_1 pt\bra_\si^M
-\sum_{j,\,\al_1+\al_2=\si}(\al_1\cdot L)\,
\bla H,\xi_j\bra_{\al_1}^M \bla \xi_j^*,pt\bra_{\al_2}^M.
  \end{eqnarray*}
\end{lemma}
\begin{proof} For present purposes we may 
think of a branched $2$-dimensional manifold as the 
realization of
a rational singular $2$-cycle formed by taking the  union of a finite number of (positively) rationally weighted oriented $2$-simplexes $(\la_i, \De_i)$, where $ \la_i\in \Q^{>0}$, appropriately identified along their boundaries. (This might have singularities  
at the vertices, but since these are codimension $2$ this does not matter.
See \cite{Mcbr,HWZ3} for a more complete description.) Since  $B$ is transverse to the strata in $\oWw_\be$,  there is a finite set $\Sing_B$ 
 of points in 
$B$ corresponding to stable maps whose domain has two components,
and the other points have domain $S^2$.  We assume that each $b\in \Sing_B$ lies in the interior of a $2$-simplex and therefore has a weight $\la_b$.  

Let
$\pi: C\to B$ be  the universal curve formed by the domains
with  evaluation map  $f:C\to M$. The marked points define two disjoint sections $s_1, s_2$ of $C\to B$, numbered so that $f\circ s_1(b) \in H, f\circ s_2(b) =x_0$. We may assume that these sections are transverse to the pullback
divisor (codimension $2$ cycle) $f^*L$.
Note that $C$ is
 the blow up of an oriented $S^2$-bundle $P\to B$ at a finite number of points, one in each fiber over $\Sing_B$. For each such $b$ we choose the exceptional sphere $E_b$ to be the component that does not contain $s_2(b)$. Since the marked points never lie at  nodal points of the domain,
the sections $s_1, s_2$ blow down to disjoint sections $s_1',s_2'$ of the $S^2$ bundle $P\to B$.  Hence $P$ can be considered as the projectivization
$\PP(V\oplus \C)$ where $V\to B$ is a line bundle and
 $s_1'= \PP(V\oplus 0)$, $s_2'= \PP(0\oplus \C)$.   Note also that $s_2=s_2'$ while $s_1$ is the blow up of $s_1'$ over those points
 $b\in \Sing_B$ for which $s_1\cap E_b\ne \emptyset$.  For such $b
 \in \Sing_B$
  set $\de_1(b): = \la_b$, and otherwise set $\de_1(b): = 0$.
 
If $B$ were a manifold then,
as in the proof of Theorem 1 in \cite[\S1.1]{LP}, we would have  
$$
\bla H,\tau_1 pt\bra_\si^M = - s_2\cdot s_2 = -\int_B c_1(V).
$$  
In our situation
we must take the weights on $B$ into account: each point $y$ in the intersection $s_2\cdot s_2$ should be given the (positive) weight $\la(\pi(y))\in \Q$
of the corresponding point $\pi(y)\in B$ as well as the sign $o(y)\in \{\pm 1\}$
of the intersection.  Thus we find
$$ 
\bla H, \tau_1 pt\bra_\si^M = -\sum_{y\in s_1\cdot s_1} o(y)\La(\pi(y))
 =: -\int_B c_1(V).
$$

We want to calculate
\begin{eqnarray*}
\bla LH, pt\bra_\si^M &=& \bla LH, pt\bra_\si^M - 
\bla H, L pt\bra_\si^M \\
&=& \int_B \ev_1^*(L) - \ev_2^*(L) \\
&=& 
\int_{s_1} f^*(L) - \int_{s_2} f^*(L)\\
&=& \sum_{y\in s_1\cdot f^*(L)} o(y)\La(\pi(y)) - \sum_{y\in s_2\cdot f^*(L)} o(y)\La(\pi(y)).
\end{eqnarray*}
This can be done just as in \cite{LP}.  The divisor
$f^*(L)$ intersects a generic fiber $F$ of $C\to B$ with multiplicity $\si\cdot L$, and  intersects the exceptional divisor $E_b$
 with multiplicity $\al_1(b)\cdot L$, where $\al_1(b)=[f_*(E_b)]$.
Since $H_2(C;\Q)$ splits as the sum $H_2(P;\Q)\oplus\sum_{b\in\Sing_B}
[E_b]\Q$,
 we may consider the difference $[s_1]-[s_2]\in H_2(C;\Q)$ 
as the sum of $[s_1']-[s_2']$ with $-\sum_{b\in \Sing_B} \de_1(b)[E_b]$.
But $[s_1']-[s_2']=k[F]$ where $k: =  -\int_B c_1(V) F$ and $[F]$ denotes the fiber class of $P\to B$.  
It follows that
\begin{eqnarray*}
\int_{s_1} f^*(L) - \int_{s_2} f^*(L) &=& 
-(\si\cdot L) \int_B c_1(V) - \sum_{b\in \Sing_B} \de_1(b)
\al_1(b)\cdot L\\
&=& (\si\cdot L) \bla H, \tau_1 pt\bra_\si^M  
-\sum_{\al_1+\al_2=\si}(\al_1\cdot L) N_{\al_1,\al_2}
\end{eqnarray*}
where $N_{\al_1,\al_2}= \sum_j
\bla H,\xi_j\bra_{\al_1}^M \bla \xi_j^*,pt\bra_{\al_2}^M$
 is the number of two-component curves, one in class $\al_1$ through $H$ and the other in class $\al_2$ through the point $x_0$.
 This completes the proof.
\end{proof}

  \NI
  {\bf Proof of Lemma~\ref{le:horpt2}}  If a symplectic $4$-manifold $(M,\om)$  is not the blow up of a rational or ruled manifold then it has a unique maximal collection $\{\eps_1,\dots,\eps_k\}$ of disjoint 
  exceptional classes (i.e. classes that may be represented by symplectically embedded spheres of self-intersection $-1$).  If we blow these down, the resulting manifold  $(\ov M, \ov\om)$ (called the minimal reduction of 
  $(M,\om)$) has trivial genus zero Gromov--Witten  invariants.  We show 
  that for these $M$  it is impossible for an invariant of the type
  $\bla pt,s\bra_\si^P$ to be nonzero.  
   
 Since  $(M,\om)$ is not strongly uniruled, it follows from
 Remark~\ref{rmk:horpt}  that $\bla pt, a\bra_\si^P=0$ 
 for all $a\in H_2(M).$
 Therefore  $\bla pt,s\bra_\si^P$ is the same for all section classes $s$.
Choose classes $\ov h_1,\ov h_2\in H^2(\ov M;\Z)$ with $\ov h_1\ov h_2\ne 0$.  Pull them back to $M$ and then extend them to $H^2(P;\Q)$ (which is possible by~\cite[Thm~1.1]{Mcq} for example.)   Multiplying them by a suitable constant
we get integral classes $h_1,h_2\in H^2(P;\Z)$ with Poincar\'e duals $H_1,H_2$.  By construction $s_0: =H_1H_2$ is a section class
and $H_j\cdot\eps_i=0$ for all exceptional divisors $\eps_i$ in $M$.

We now claim that we can apply Lemma~\ref{le:LP} to evaluate
the nonzero invariant  $\bla H_1H_2, pt\bra_\si^P$.  
For this, it suffices to show that the space of (regularized) stable maps in class $\si$ and through the point $x_0$ can be constructed as a branched $2$-manifold $B'$.
To this end,
 consider an $\Om$-tame almost complex structure $J_P$ on $(P,\Om)$
for which the projection $\pi:(P,J_P)\to (S^2,j)$ is holomorphic.
Then $J_P$ restricts on each fiber $P_z: = \pi^{-1}(z)$ to an $\om$-tame almost complex structure on $M$.  
Every $J_P$-holomorphic stable map in class $\si$ consists of a holomorphic section plus some fiberwise bubbles. 
Since the family of such stable maps through  some fixed point $x_0$ has real dimension $2$, we can assume that each such bubble is a $k$-fold cover of an embedded regular curve in some class $\be$ with $c_1(\be)\ge 0$ and $\om(k\be)\le \ka = \Om(\si)$, and that the sections through $x_0$ with energy $\le \ka$ are regular and so lie in classes with $2\le c_1(\si')\le 3$. 
For any $J_P$, each exceptional
 class $\eps_i$ is represented by a unique
embedded sphere. Since $(M,\om)$ is not rational or ruled, it follows from Liu \cite{Liu}
 that these  are the only 
$J_P$-holomorphic spheres in classes $\be$ with $c_1(\be)>0$. (For details of this argument see ~\cite[Cor~1.5]{MS0}.)  

Since regular sections through $x_0$ in classes with $c_1(\si')=2$
are isolated, there can be a finite number of two component stable maps whose bubble is an exceptional sphere and there
 are no stable maps involving
multiply covered exceptional classes.  
However there may be some
with multiply covered bubbles in classes $\be$ with $c_1(\be) = 0$.  
 To deal with these, choose a suitable very small multivalued perturbation $\nu$ over the moduli spaces of
fiberwise curves with class $k\be$ for $\om(k\be)\le \ka$
so that there are only isolated solutions of the corresponding perturbed equation.  
Since the sections through $x_0$ are regular, they can meet one of these isolated bubbles only if 
they lie in a moduli space of real dimension $2$ and for generic choices 
of $J_P$ and $\nu$ they will meet only one such bubble.  Therefore for this choice of $\nu$ the perturbed moduli space contains isolated two-component 
stable maps. It remains to extend the perturbation over a neighborhood
of this stratum in $\oWw_\si$ (tapering it off to zero outside this neighborhood), and to define $B'$ as the  solution space of the resulting perturbed Cauchy--Riemann equation.

 Lemma~\ref{le:LP} now implies that
\begin{eqnarray*}
\bla H_1H_2, pt\bra_\si^P &=& (\si\cdot H_1)\,\bla H_2,\tau_1 pt\bra_\si^P
-\sum_{j,\al_1+\al_2=\si}(\al_1\cdot H_1)\,
\bla H_2,\xi_j\bra_{\al_1}^P \bla \xi_j^*,pt\bra_{\al_2}^P,
  \end{eqnarray*}
 where $\xi_j$ runs over a basis for $H_*(P)$  with dual basis  $\xi_j^*$.
 Note that we may choose this basis so that exactly one of  each pair $\xi_j,\xi_j^*$
 is a fiber class.
 
The first term must vanish, since otherwise
$(M,\om)$ is strongly uniruled by Lemma~\ref{le:horpt}.  
Therefore there must be some nonzero product.  If $\al_2$ is a fiber class then by Lemma~\ref{le:fib} $ \bla \xi_j^*,pt\bra_{\al_2}^P=
 \bla \xi_j^*\cap M,pt\bra_{\al_2}^M$ so that  $(M,\om)$ is strongly uniruled by definition.  Hence these product terms vanish.  
 Further if $\al_2$ is a section class and $\xi_j^*\ne [M]$ is a fiber class then $(M,\om)$ is strongly uniruled by Remark~\ref{rmk:horpt}.
  On the other hand if $\xi_j^*=[M]$ then there is a nonzero invariant
$\bla H_2,s\bra_{\al_1}^P$ for $\al_1\in H_2(M)$ and some section class $s$ which implies that $c_1^M(\al_1) = 1$.  
Also because $\al_1\cdot H_2\ne 0$, the choice of $H_2$ implies that $\al_1$ is not
one of the exceptional classes $\eps_i$. But this is impossible 
since we saw above that the only  bubbles
with Chern class $1$ are the exceptional spheres.
 Therefore there must be a nonzero term in which both 
$\al_1$ and $\xi_j$  are fiber classes.
 In this case,
 $\bla H_2,\xi_j\bra_{\al_1}^P=
\bla H_2,\xi_j\cap M\bra_{\al_1}^M$ is an invariant in $M$.  But the only nonzero classes $\al_1\in H_2(M)$ with nontrivial $2$-point invariants are the exceptional divisors $\eps_i$ and
$\eps_i\cdot H_1=0$ by construction.  Therefore these terms must vanish as well.  This completes the proof.\QED

\section{Special cases}\labell{s:fur}

We now discuss some special $S^1$-actions for which it is possible to prove directly that $(M,\om)$ is strongly uniruled.

\begin{prop}\labell{prop:sfree}  Suppose that $(M,\om)$ is a semifree Hamiltonian $S^1$-manifold. Then $(M,\om)$ is  strongly
uniruled.
\end{prop} 

\begin{proof}  Denote by $\ga$ the element in $\pi_1(\Ham M)$ represented by the circle action and by $\Ss(\ga)\in QH_*(M)^{\times}$ its Seidel element.  Since the action is semifree, 
\cite[Thm~1.15]{MT} shows that 
$$
\Ss(\ga)*pt = a\otimes q^{-d} t^\ka + \mbox { l.o.t.}
$$
 where $a$ is a nonzero element of $H_{2d}(M)$ with $d >0$.
 But if  $(M,\om)$ is not strongly
uniruled, $\Ss(\ga)*pt = (\1\otimes \la + x)*pt = pt\otimes \la$ by
 Lemma~\ref{le:QH}. An alternative proof is given in 
 Proposition~\ref{prop:al}.
\end{proof}


The ideas of \cite{MT} also work when the isotropy weights have absolute value 
$\le 2$.  (In this case, we say that the action has at most $2$-fold isotropy.)  This property is stable under blow up along the maximal or minimal fixed point sets: cf. Lemma~\ref{le:semi}. 

\begin{prop} \labell{prop:2iso} Suppose that $(M,\om)$ is an effective Hamiltonian $S^1$-manifold with at most $2$-fold isotopy and isotropy weights $1$ along $F_{\max}$. Then $(M,\om)$ is strongly  
uniruled.
\end{prop} 

\begin{proof} 
By Proposition~\ref{prop:max}
 $$
 \Ss(\ga) = a_{0}\otimes q^d\, t^{K_{\max}} + \sum_{\be\in H_2(M;\Z),\,\om(\be)>0} a_\be\otimes q^{m-c_1(\be)}\,t^{K_{\max}-\om(\be)},
 $$
where  $a_\be \in H_*(M)$ and $a_0$ is in the image of $H_*(F_{\max})$ in $H_*(M)$.
Suppose that $(M,\om)$ is not strongly  uniruled.
By Lemma~\ref{le:QH} (iii), there is at least one term 
in $ \Ss(\ga)$ with $a_\be  = r\1$ where $r\ne 0$.  Consider the term of this form with minimal $\om (\be )$.  Let $J$ be a generic $\om$-tame and $S^1$-invariant almost complex 
structure on $M$, with corresponding metric $g_J$.
Proposition 3.4 of \cite{MT} shows that
 in order for $r\ne 0$ there 
must be, for every point $y\in F_{\min}$,  an $S^1$-invariant 
$J$-holomorphic genus zero stable map in class $\be $ that intersects $F_{\max}$ and $y$.  Such an invariant element consists of a connected 
string of $2$-spheres from $F_{\max}$ to $y$, possibly with added bubbles. Components of the string (called {\it beads} in \cite{MT}) either
lie in the fixed point set $M^{S^1}$ or are
 formed by the orbits of the $g_J$-gradient trajectories of  $K$.  The  
 energy $\om(\be')$ of an invariant sphere in class $\be'$
 that joins the two fixed point components $F_1, F_2$ is at least $|K(F_1)-K(F_2)|/q$,
  where $q$ is the order of the isotropy at a generic point of the sphere.  
 Therefore, if the isotropy has order at most $2$
 the energy needed to get from $F_{\max}$ to 
 $y\in F_{min}$ is
 at least $(K_{\max}-K_{\min})/2$.  In the case at hand, it is strictly larger than $(K_{\max}-K_{\min})/2$ since the 
 first  element of the string has trivial isotropy.
 Therefore there is $ r\ne 0, x\in \Qq_{-}$ such that
 $$
 \Ss(\ga) =  \1\otimes 
 \bigl(rt^{K_{\max}-\ka} + \mbox{ l.o.t}\bigr) + x,\quad \ka > (K_{\max}-K_{\min})/2.
 $$
Similarly, there is $ r'\ne 0,  x'\in \Qq_{-}$ such that
$$
 \Ss(\ga^{-1}) =\1\otimes 
 \bigl(r't^{-K_{\min}-\ka'} + \mbox{ l.o.t}\bigr) + x',\quad \ka' > (K_{\max}-K_{\min})/2.
 $$
Since $\Ss$ is a homomorphism, we know that 
$ \Ss(\ga)* \Ss(\ga^{-1})=\1$.  Now assume that $M$ is not strongly uniruled.  Then by Lemma~\ref{le:QH}(ii) the above expressions imply that
$$
 \Ss(\ga)* \Ss(\ga^{-1}) = rr'\1\otimes \bigl(t^\de + 
 \mbox{ l.o.t}\bigr),\quad \de<0,
 $$
 a contradiction.
\end{proof}

The previous results give conditions under which $(M,\om)$ is strongly uniruled,
but they do not claim that the specific invariant $\bla pt\bra^M_\al$ is nonzero, where $\al$ is the orbit of a generic gradient flow line from $F_{\max}$
 to $F_{\min}$.  There are two cases when we can prove this. Note that condition 
 (ii) is not very general since $F_{\max}$ is often obtained by blow up 
 and any such manifold is uniruled.

\begin{prop}\labell{prop:al}
Suppose that $(M,\om)$ is a Hamiltonian $S^1$-manifold whose maximal and minimal fixed point sets are  divisors.   Suppose further that at least one of the following conditions holds:\SSS

\NI {\rm  (i)} the action is semifree, or \SSS

\NI {\rm (ii) } there is an $\om$-tame almost complex structure  
$J$ on $F_{\max}$ such that the nonconstant $J$-holomorphic 
spheres in $F_{\max}$ do not go through every point.\SSS

\NI
Then $\bla pt\bra^M_\al\ne 0.$
\end{prop}
\begin{proof}  Suppose first that the action is semifree and let $J$ be a 
generic $S^1$-invariant almost complex structure on $M$.  
Then by \cite[Lemma~4.5]{MT} the gradient flow of the moment map with respect to the 
associated metric $g_J(\cdot,\cdot) =\om(\cdot,J\cdot)$ is Morse--Smale.  Hence 
for a generic point $x_0$ of $F_{\max}$ all the gradient flow lines that start 
at $x_0$  end on
$F_{\min}$.   The union of these flow lines is an invariant $J$-holomorphic
$\al$-sphere through $x_0$.  Moreover there is no other invariant 
$J$-holomorphic stable map in class $\al$ through $x_0$.  For as in the previous proof this would have to consist of a sphere $C$ through $x_0$ in $F_{\max}$ together with a  string of $2$-spheres from a point $x$ in $F_{\max}$ to a point $y$ in $F_{\min}$, possibly with added bubbles.  (The string has to reach $F_{\min}$ since $\al\cdot F_{\min}= 1$.)  But because the action is semifree
 the energy of such a string is at least $\om(\al)$.  Since $\om(C)>0$ this is impossible.
Therefore there is only one  invariant $J$-holomorphic
stable map in class $\al$  through $x_0$.  Since this is regular,  $\bla pt\bra^M_\al=1$.

A similar argument works in case (ii). Choose $J$ to be a generic  $S^1$-invariant extension of the given almost complex structure on $F_{\max}$.  The arguments of \cite[\S4.1]{MT} show that the set $X$ of points in  $F_{\max}$ that flow down to
some intermediate fixed point set of $K$ is closed and of codimension at least $2$.  Moreover by perturbing $J$ in $M\less F_{\max}$ we may jiggle $X$ so
that there is a point $x_0\in F_{\max}\less X$ that does not lie on any
 $J$-holomorphic sphere in $F_{\max}$.  Hence again there is only one invariant
stable map in class $\al$  through $x_0$.  The result follows as before.
\end{proof}

We end by discussing the case when 
$H^*(M;\Q)$ is generated by $H^2(M)$.   Our main result here is the following.

\begin{prop}\labell{prop:usu} Assume that $H^*(M;\Q)$ is generated by $H^2(M)$.  Then $(M,\om)$ is strongly uniruled iff it is uniruled.
\end{prop}

%

The proof is given  below.  By Theorem~\ref{thm:main} we  immediately
obtain:

\begin{cor}\labell{cor:coh}  Suppose that $(M,\om)$ is a 
Hamiltonian $S^1$ manifold such that  
$H^*(M;\Q)$ is generated by $H^2(M)$.
Then $(M,\om)$  is strongly uniruled.
\end{cor}

\begin{rmk}\rm  (i) Observe that if $M$ is a Hamiltonian $S^1$-manifold
such that
$H^*(M;\Q)$ is generated by $H^2(M)$ then the same holds for the blow up of $M$ along any of its fixed point submanifolds $F$.  For because the moment map $K$ is a perfect Morse function the inclusion $F\to M$ induces an injection on homology.  (Any class $c$ in $H_*(F)$ can be written as $c^-\cap F^+$, where $c^-, F^+$ are the canonical downward and upward extensions of $c, [F]$ defined for example in ~\cite[\S4.1]{MT}.)
  Hence $H^*(F)$ is generated by the restrictions 
  to $F$ of the classes in $H^2(M)$.  
   Since the exceptional divisor $E$ is a
    $\PP^k$ bundle over $F$, $H^*(E)$  is also generated by $H^2(E)$: in fact the generators are the pullbacks of the classes in 
    $H^2(F)$ plus the first Chern class 
   $c\in H^2(E)$ of the canonical line bundle over $E$.   
   But $c$ is the restriction to $E$ of the class $\Tc$ in 
   $\TM$ that is Poincar\'e dual to $E$.  It follows easily (using the 
   Mayer-Vietoris sequence) that $H^*(\TM)$ is generated by 
   the pullback of the classes in $H^2(M)$ together with $\Tc$;
    see the discussion of the cohomology  of a 
   blow up given in \cite[\S5.1]{HLR}.\SSS
   
\NI (ii)    It is not clear which Hamiltonian $S^1$ manifolds have the property that $H^*(M)$ is generated by $H^2(M)$.
It is not enough that the fixed points are isolated.
For example, Sue Tolman\footnote{Private communication.} pointed out that the complex Grassmannian $Gr(2,4)$
of $2$-planes in $\C^4$ has $H^2$ of dimension $1$ and $H^4$ of dimension $2$ so that $H^4\ne (H^2)^2$. It also has an $S^1$ action with precisely $6$ fixed points.  However, it is enough to have isolated fixed points plus semifree action, since in this case
Tolman--Weitsman~\cite{TW} show that
  $H^*(M)$ is isomorphic as a ring to the cohomology of a product of $2$-spheres.
Also $H^*(M)$ is generated by $H^2(M)$ in the toric case.  However, 
 these cases are uninteresting in the present context since we already
know that these manifolds are  strongly uniruled, in the former case by Proposition~\ref{prop:al} and in the latter by the fact that
 toric manifolds are projectively uniruled.\end{rmk}

The proof  of Proposition~\ref{prop:usu} uses 
the identities (\ref{eq:LPi}) and (\ref{eq:LPii}).
 Note that each time we apply one of these formulas 
we must take special care with the zero class. 
The following lemma is well known: cf. \cite[Lemma~4.7]{HLR}.
\MS

\begin{lemma}\labell{le:00}
  $\bla\tau_{k_1} a_1, \dots,\tau_{k_p} a_p\bra^M_{0}\ne 0$ for some $p\ge 3$
 only if the intersection product of the classes $a_i$ is nonzero (i.e. the (real) codimensions of the $a_i$ sum to $2n$) and $\sum k_i = p-3$.
 \end{lemma}
 \begin{proof}  This is immediate from the definition if $k_i=0$ for all $i$. To prove the general case, one can either argue  directly or
can construct an inductive proof based on the identity (\ref{eq:LPi}). 
To understand why the invariant vanishes when 
$\dim(a_1\cap \dots\cap a_m)>0$, observe that in 
 this case the moduli space $\Mm_0$ of constant maps
with {\it fixed} marked points and through the constraints can be identified with
$a_1\cap \dots\cap a_m$ and so has dimension $>0$.   But the classes $\psi_i$ are trivial on $\Mm_0$ and hence the integral of any product of the
 $\psi_i$ over the full moduli space (with varying marked points) must vanish.
 \end{proof}  
\MS

\NI {\bf Proof of Proposition~\ref{prop:usu}.}

Since any strongly uniruled manifold is uniruled,
it suffices to prove the converse.  This in turn is an 
immediate consequence of the next lemma.

\begin{lemma}\labell{le:GWsu} Suppose that $H^*(M;\Q)$ is generated by $H^2(M)$ and that there is a nonzero invariant of the form
\begin{equation}\labell{eq:su}
\bla \tau_{k_1} pt, \tau_{k_2}a_2, \dots,\tau_{k_m} a_m\bra^{M}_{\be},\quad a_i\in H_*(M), k_i\ge 0, \be\ne0.
\end{equation}
Then $(M,\om)$ is strongly uniruled.
\end{lemma}

\begin{proof}
The first two steps in this argument apply to all $M$ and 
are contained in the proof of ~\cite[Thm~4.9]{HLR}.  We include them for completeness.
Without loss of generality we consider a nonzero invariant  
(\ref{eq:su})
such that $m$ is minimal and $\om(\be)$ 
is minimal among all nonzero invariants (\ref{eq:su}) of length $m$ with $\be\ne 0$. We then order the indices $k_i$ so that $k_2\le \dots\le k_m$
and suppose that the $k_i$ for $i>1$ are minimal with respect to the lexicographic order for the given $m,\om(\be)$. 
Finally we choose a minimal $k_1$ for the given $m, \om(\be),$ and $k_i,i>1$.
\MS

\NI {\bf Step 1:} {\it We may assume that  $k_i=0$ for $i>1$.}
 
  If not,  let $r$ be the minimal integer greater than one such that $k_r\ne 0$.  Suppose first that $m\ge 3$ and $r>2$ and apply 
  (\ref{eq:LPi}) with $i=r, j=1$ and $k=2$.  Then the invariant  (\ref{eq:su}) is a sum of products
$$
\bla\tau_{k_r-1} a_r, \xi,\dots\bra^M_{\be_1} \;
\bla \xi^*, \tau_{k_1}pt, a_2,\dots\bra^{M}_{\be-\be_1},
$$
 where $\xi$ runs over a basis for $H_*(M)$ with dual basis 
$\{\xi^*\}$, and
the dots represent the other constraints $\tau_{k_\ell}a_\ell$ (which may be distributed in any way.)  There must be a nonzero product of this form.

We now show that this is impossible. Suppose first that there is
 such a product with $\om(\be_1)>0$.  Then the second factor is an invariant of type   (\ref{eq:su}) in a class $\be'$ with $\om(\be')<\om(\be)$ and at most $m$ constraints.  Since our assumptions imply that all such terms vanish, this is impossible.  Hence any nonzero product
 must have $\om(\be_1) = 0$ 
and hence  $\be_1=0$.  But then the second factor 
has at least one fewer nonzero $k_i$, since it has the homological
constraint $\xi^*$ instead of $\tau_{k_r}a_r$.  This contradicts the
assumed minimality of $k_2,\dots,k_m$.

This completes the proof when $r>2$.  If $m\ge 3$ but  $r=2$ use the same argument but take $k=3$ instead of $k=2$.  If $m<3$  add 
divisorial constraints
to get a nonzero invariant with $3$ constraints.  The reader can check that the previous argument still goes through because it does not use the minimality of $m$ in any essential way. 
\MS

\NI {\bf Step 2:} {\it  $k_1=0$.}

If $k_1>0$ apply  (\ref{eq:LPi}) with $i=1, j=2$ and $k=3$.
Again, there must be a nonzero product of the form
$$
\bla\tau_{k_1-1} pt, \xi,\dots\bra^M_{\be_1} \;
\bla \xi^*, a_2, a_3,\dots\bra^{M}_{\be-\be_1}.
$$
Since the first factor can have at most $(m-1)$ constraints, the minimality of $m$ implies that $\be_1=0$.   But then $pt\cap \xi\ne 0$, 
so that $\xi=[M]$.    Hence $\xi^*=pt$ and the second factor is an invariant of the required form  with $k_1=0$.\MS

\NI {\bf Step 3:} {\it Completion of the proof.}

By  hypothesis on $M$ and Step 2, there is a nonzero invariant
$\bla pt, H_2^{i_2}, \dots,H_m^{i_m}\bra^M_{\be}$ with $m$ constraints, where $H_j\in H_{2n-2}(M)$. 
Moreover, we may assume that all invariants (\ref{eq:su}) in a class $\be'$ with $\om(\be')<\om(\be)$ or $m'<m$ vanish and that the set $i_2\le \dots\le i_m$ is minimal in the lexicographic ordering.
Note that $i_2>1$ since otherwise we can reduce $m$ by using the divisor equation. 
 We must show that $m\le 3$.  
 
 Suppose not and apply (\ref{eq:LPii}) with  $i=2$ and $j=1$.  Since $H_2\cap pt = 0$ the first term on the
 RHS vanishes.  The second is a multiple of
 $\bla\tau pt, H_2^{i_2-1},H_3^{i_3},\dots, H_m^{i_m}\bra^M_{\be}$.
 Suppose it is nonzero and
 apply  (\ref{eq:LPi}) to it, with $i=1$ and $ j,k,=2,3$.
This gives  a sum of terms
 $\bla pt, \xi,\dots\bra^M_{\be_1}\; 
\bla  H_2^{i_2-1}, H_3^{i_3}, \xi^*,\dots\bra^{M}_{\be-\be_1}$.
Since the first factor has $<m$ constraints, we must have $\be_1=0$. 
But then $\xi=[M]$ so that $\xi^* = pt$.  Hence all the other constraints must lie in the second factor (since it must have at least $m$ constraints).
Therefore the first factor 
is a constant map with only two constraints, i.e. it  is unstable.  But this 
 is not allowed.
  Therefore, the second term in (\ref{eq:LPii}) must vanish.

It remains to consider the product terms in (\ref{eq:LPii}), namely
 $$
(\be_1\cdot H_2)\;\bla H_2^{i_2-1}, \xi,\dots\bra^M_{\be_1}\; 
\bla  pt, \xi^*,\dots\bra^{M}_{\be_2},
$$ 
where $\be_1+\be_2=\be$.
In any nonzero product of this form,  $\be_1\ne 0$. Also 
 the stability condition on the second factor implies that if $\be_2=0$ 
 this term must have another constraint.  Since this must have  
the form $H_j^{i_j}$ with $i_j>0$,  this is not
  possible by Lemma~\ref{le:00}. 
 Therefore $0<\om(\be_2)<\om(\be)$.  But then the second factor
vanishes by the minimality of $\om(\be)$.
\end{proof}

  \begin{rmk}\rm  Suppose that $(M,\om)$ is uniruled with even constraints,
  i.e. there is a nonzero invariant of the form
(\ref{eq:su}) in which all the $a_i$ have even degree.  Then the proof of
Lemma~\ref{le:GWsu} goes through if we assume only that the even part
$H^{ev}(M)$ of the cohomology ring
 is generated by $H^2$.  For if  the $a_i$ have even degree, odd dimensional homology classes appear
in the above proof only as elements $\xi, \xi^*$. Since these always appear as part of invariants where all the other insertions have even dimension, all terms involving odd dimensional $\xi, \xi^*$  must vanish. The appendix contains other results about such manifolds;
cf. Propositions~\ref{prop:frob} and \ref{prop:a}.
\end{rmk}

\begin{appendix}
  \section{The structure of $QH_*(M)$ for uniruled $M$}

  In this appendix we explore the extent to which the uniruled property can be seen in quantum homology, completing the discussion begun in 
  Lemma~\ref{le:QH}.   To simplify we shall ignore contributions to the quantum product from the odd dimensional homology classes.  Hence our results are not as general as they might be.

  Let $\FF: = \La$ be the field $\La^{\univ}$ of generalized Laurent series.  Observe that the
 even quantum homology 
 $$
 QH_*^{ev}(M): = 
 \bigoplus_{i=0}^n H_{2i}(M;\R)\otimes \La[q,q^{-1}]
 $$
 is a subring of $QH_*(M)$ because
when $a,b\in H_*(M)$ have even degree $\bla a,b,c\bra^M_\be=0$ unless $c$ also has even degree. We shall denote by
 $\Aa$ its subring $QH_{2n}(M) = QH_{2n}^{ev}(M)$ of elements of degree $2n$, regarded as a commutative algebra over $\FF$.  (Equivalently we can think of $\Aa$ as the algebra obtained 
from $QH_{ev}(M)$ by setting $q=1$.)

  Choose a basis $\xi_i$ for $H_{ev}(M;\Q)$ with $\xi_0=pt$, $\xi_N=\1$
  and so that $0<\deg(\xi_i)<2n$ for the other $i$.  These elements form a finite basis for $\Aa$ considered as a  vector space over $\FF$.  Hence there is a well defined linear map $f:\Aa\to \FF$  given by
  $$
 f(\sum_{i=0}^N \la_i\xi_i) = \la_0.
 $$
 Since the corresponding pairing  $(a,b):= f(ab)$ on $\Aa$ is nondegenerate, 
  $(\Aa,f)$ is a commutative Frobenius algebra.\footnote
  {
  A pair $(\Aa, f)$ consisting of a commutative finite dimensional unital algebra together with a linear functional $f:\Aa\to\FF$ satisfies the Frobenius nondegeneracy condition iff 
  $\ker f$ contains no nontrivial ideals.}
It is easy to see that it satisfies the other conditions of the next lemma, with $p=pt$ and $\Mm = {\rm span\,}\{\xi_i: i\ne 0,N\}$.  In particular $f(pa) = 0$ when $a\in \Qq_-: = p \FF \oplus
 \Mm$ because all 
Gromov--Witten invariants of the form
$$
\bla a,b,[M]\bra^M_{\be}, \quad \be\ne 0, \;a,b\in H_{<2n}(M)
$$
vanish.

 \begin{lemma}\labell{le:frob}  Let $(\Aa,f)$ be a finite dimensional commutative Frobenius algebra over a field $\FF$ that decomposes
 additively  as
$\Aa = p\,\FF\oplus \Mm \oplus \1\,\FF$ where
$f(p)=1$ and $\ker f = \Mm\oplus \1\,\FF$. Suppose further that
$f(pa) = 0$ for all $a\in \Qq_-: =p\,\FF\oplus \Mm$.   Then
$pa=0$ for all $a\in \Qq_-$ iff
there are no units in $\Qq_-$.
\end{lemma}
\begin{proof}
One implication here is obvious: if $a\in \Qq_-$ is a unit then
$pa\ne 0$.   Conversely, suppose that $pa\ne 0$  
for some $a\in\Qq_-$, but that there are no units in $\Qq_-$.
We shall show by a sequence of steps that there are no Frobenius algebras 
$(\Aa,f)$ that have this property as well as satisfying the other conditions 
 in the statement of the lemma.
 
    Decompose $\Aa$ as a sum
$\Aa_1\oplus\dots\oplus \Aa_k$ of indecomposables and let $e_1,\dots, e_k$ be the corresponding unitpotents.   Thus  for all $i,j$
$$
e_i e_j=\de_{ij}e_i,\quad\mbox{and }\;
\Aa_i = \Aa e_i.
$$
     Each $e_i$ may be written uniquely in the form $\la_i\1 + x_i$ where $x_i\in \Qq_-$ and $\la_i\in \FF$.  Order the $e_i$ so that
the nonzero $\la_i$ are $\la_1,\dots,\la_\ell$.
  Since $\sum e_i = \1$, we must have 
$\sum _{i=1}^\ell\la_i=1$.  In particular, $\ell\ge 1$.  \MS

 \NI {\bf Step 1:} 
{\it All units in $\Aa_i, i>1,$ lie in $\Qq_-$. Hence $\ell=1$.}

Suppose there is a unit in $\Aa_i$ of the form $\mu_i\1 + x$, where $\mu_i\ne 0$  and $i>1$.  Choose nonzero $\nu_j\in \La$ for $j\ne i, j\le \ell$, so that $\sum \nu_j\la_j=\mu_i$ and set
$\nu_j: = 1 $ when $ j>\ell$.
Then $u: = \sum_{j\ne i} \nu_je_j - u_i$ is a unit of $\Aa$ since it is a sum of units, one from each factor.  By construction, the coefficient $\1$ in $u$ vanishes.  Hence $u$ 
 is a unit of $\Aa$ lying  in $\Qq_-$, which, by hypothesis, is impossible.
\SSS

\NI {\bf Step 2:}  {\it All nilpotent elements 
in $\Aa$ lie in $\Qq_-$.}
  
If $n = \1-x$  is nilpotent for some  $x\in \Qq_-$, then $x = \1-n$ is a
 unit of $\Aa$ lying in $\Qq_-$.

\NI {\bf Step 3:} {\it For all $i>1$, each $\Aa_i\subset \Qq_-$ and $pe_i=0$.}

The standard theory of indecomposable finite dimensional algebras over a  field implies that every nonzero element in 
$\Aa_i$ either is a unit or
 is nilpotent; see Curtis--Reiner \cite[pp.370-2]{CR}. 
The units in $\Aa_i, i>1$, lie in $\Qq_-$ by Step 1, and the nilpotent elements do too by Step 2.  This proves
 the first statement.  Thus $e_ix\in \Qq_-$ for all $x\in \Aa$ and $i>1$,
so that by our initial assumptions   $f(pe_ix) = 0$
for all $x$.  But the restriction of $f$ to each summand $\Aa_i$ 
is nondegenerate. Hence this is possible only if $pe_i=0$.
\SSS

\NI {\bf Step 4:} {\it Completion of the argument.}

Let $e = \sum_{i>1} e_i$. By the above we may assume that
 $p = p(\1-e)\in \Aa_1$. 
Decompose $\Aa_1$ additively as
the direct sum
 $\Nn \oplus \Uu$ where $\Nn$ is the subspace formed by the 
nilpotent elements and $\Uu$ is a 
complementary subspace.  Step 2 implies that $\Nn\subset 
\Aa_1\cap\Qq_-$.  On the other hand, as in Step 3,
 any nonzero element in 
$\Aa_1\cap\Qq_-$ that is not nilpotent
is a unit $u$ in $\Aa_1$.  If any such existed then
$u+e$ would be a unit of $\Aa$ lying in $\Qq_-$.  Hence by hypothesis we must have 
$\Nn=\Aa_1\cap \Qq_-$. Thus $p$ is nilpotent and $\dim \Uu = 1$, spanned by $e_1$.
Further because $\Qq_-$ is spanned by $\Nn$ and the $\Aa_i, i>1$, 
$p$  does not annihilate 
all elements in $\Nn$.  

We now need to use further information about the structure of $\Aa_1$.
Recall that the socle $\Ss$ of an algebra $\Aa_1$ is the annihilator   of $\Nn$.  Therefore, our assumption on $p$ implies that $p\notin
\Ss$.  But 
there always is $w\in \Nn$ such that $pw$ is a nonzero element of $\Ss$.
To see this, choose a set $x_1=p, x_2,\dots,x_k$ of multiplicative generators for $\Nn$ that includes $p$.  There is $N$ such that all products of the $x_i$ of length $>N$ must vanish. (Take $N$ to be the sum of the orders of the $x_i$.) Therefore there is  a nonzero product 
 of maximal length that contains $p$ and so can be written as $pw$. Since $pwx_i=0$ for all $i$, $pw\in \Ss$. (A more precise version
 of this argument is given in Abrams~\cite[Prop.3.3]{Ab2}.)


The argument is now quickly completed.  For, by the Frobenius nondegeneracy condition there must be $z\in \Aa$ such that  
 $(pw,z)=f(pwz)\ne 0$.  By Step 3
 we may assume that $z\in \Aa_1$ so that $z=\la e_1 +n$ for some $n\in \Nn$.  But then $pwz=pw(\la e_1)=\la pw$ so that $f(\la pw)=\la f(pw)\ne 0$.  But $w\in \Nn\subset \Qq_-$ by construction. Hence 
 $f(pw)=0$ by our initial assumptions.  This contradiction shows that
our assumption that $\Qq_-$ has no units must be wrong. 
  \end{proof}

 We shall say that $(M,\om)$ is {\bf  uniruled with even constraints} 
if there is a nonzero Gromov--Witten invariant
of the form $\bla pt,a_2,\dots,a_k\bra^M_\be$ with $\be\ne 0$ and all $a_i$ of even degree.  Similarly,
 $(M,\om)$ is {\bf strongly uniruled with even constraints} if there is a nonzero invariant of this kind with $k=3$.   For example, 
 \cite[Cor.~4.3]{HLR}
shows  that any
 projective manifold  that is uniruled is in fact  strongly uniruled with even constraints.

\begin{prop}\labell{prop:frob} $(M,\om)$ is
 strongly uniruled with even constraints
 iff the even quantum homology ring  $QH_*^{ev}(M)$  has a unit in $\Qq_-$.
\end{prop}  
\begin{proof}  
Note that  $QH_*^{ev}(M)$  has a unit iff its degree $2n$ part
$\Aa$ has.  Also the subring $\Qq_-(\Aa)$ of $\Aa$ considered above is just the intersection $\Qq_-\cap QH_{2n}^{ev}(M)$.
Therefore this is an immediate consequence of Lemma~\ref{le:frob}.
\end{proof}

If $(M,\om)$ is uniruled rather than strongly uniruled we can still see some effect on quantum homology if instead of considering the small quantum product $*$ we consider the whole family of products $*_a$, $a\in \Hh$. 
Here we shall take $\Hh: =H_{ev}(M;\C)$ and correspondingly allow the coefficients $r_i$ of the elements $\sum r_it^{\ka_i}$ in $\La$
to be in $\C$.
Let $\xi_i, i=0,\dots,N,$  be a basis for $\Hh$ with $\xi_0=pt$ as before, and
identify $\Hh$ with $\C^{N+1}$ by thinking of this as the standard basis in $\C^{N+1}$.
Denote by $T_0,\dots, T_N$ the corresponding coordinate functions on $\Hh$, thought of as formal variables.  If
$\al = (\al_1,\dots,\al_p)$ is a multi-index with $0\le \al_i\le N$
 define
$$
\xi^\al: = (\xi_{\al_1},\dots,\xi_{\al_p})\in \Hh^{p},\quad
T^\al: = \prod_{i=1}^pT_{\al_i}.
$$
In this language, 
the (even) Gromov-Witten potential
$\Phi(t,T)$ is the formal power series in the variables $t$ and $T$
given by 
$$
\Phi(t,T): = \sum_{\al} \sum_\be 
\frac 1{|\al|!}\bla \xi^\al\bra_{\be}\; t^{-\om(\be)}T^{\al}.
$$
Let us assume\footnote
{
We make this assumption  to simplify our subsequent  discussion.
It is satisfied in the case of manifolds such as $\C P^n$ (cf. \cite[Ch~7.5]{MS}), but
there is at present little understanding of when it is satisfied
in general.  Even if it were not satisfied, one could presumably adapt the results below in any particular case of interest.} 
 that the following condition holds:\SSS

\NI {\bf Condition (*):}\,  {\it there is $\de>0$ such that
this series converges if we consider the $T_i$ to be complex numbers such that $|T_i|\le \de$.}\SSS

Then we may think of $\Phi$ as a function defined near $0\in \Hh$ with values in the field $\La=\La_{\C}^{univ}$.
Further, as explained for example in \cite[Ch~11.5]{MS}, 
 the structure constants of the  associative product
$x*_ay$ are given by evaluating the third derivatives 
 of $\Phi(t,T)$ with respect to the variables $T_i$ at the point $T=a$.  In other words,\footnote
 {One needs to interpret these formulas with some care.  When we set $T=a$ we are thinking of $T$ as the set of numbers  $(T_0,\dots,T_k)$
 that are the coordinates of $a=\sum T_\ell\xi_\ell$. On the other hand, the arguments of a Gromov--Witten invariant are homology classes. Thus $\bla a,\dots,a\bra^M_{p,\be} = \sum_{\al}
 \bla \xi^\al\bra^M_{p,\be}t^{-\om(\be)}T^\al|_{T=a}$.}
 $\xi_i *_a\xi_j =\sum_k c_{ij}^k(a) \xi_k^*,
 $
 where 
\begin{eqnarray}\labell{eq:A1}
 c_{ij}^k(a) &=& \frac{\p^3 \Phi(t,T)}{\p T_i \p T_j \p T_k}  \Big|_{T=a}=
\sum_m\frac1{m!}\; \bla \xi_i,\xi_j,\xi_k, a,\dots,a
 \bra^M_{m+3,\be}\, t^{-\be}.
\end{eqnarray}
 We denote the corresponding (ungraded)  rings by $QH_{ev}^{a}(M)$. 
As before, they are Frobenius algebras over the field $\FF: = \La$.

The main point for us is the following lemma:

\begin{lemma}  Assume that condition (*) holds and that there is a nonzero invariant
 $\bla pt,a_2,\dots,a_k\bra^M_\be$ where $\be\ne 0$
  and the $a_i$ have even degree.  Then there are 
 $a,b\in \Hh$ with $\deg b < 2n$
  such that  $pt*_a b\ne 0$.
 \end{lemma}
 \begin{proof}  Let $m+3$ be the minimal $k$ for which some 
  $\bla pt,a_2,\dots,a_k\bra^M_{k,\be}$
  does not vanish. 
  If $m\le 0$ then we can take 
 $a=\1$
 so that, by Lemma~\ref{le:00},  $*_a$ is the usual product.  If $m>0$ then the hypothesis implies that some invariant of the form
$\bla pt,\xi_j,\xi_k, a,\dots,a
 \bra^M_{m+3,\be}$ is nonzero.  
 (This holds because GW invariants are symmetric and multilinear functions of their arguments.)
 But by equation (\ref{eq:A1})
 the coefficients
  $c_{ij}^k(a)$ are power series in the coordinates of $a\in \Hh$.
  Hence  the fact that one coefficient of the power series for
 $c_{0j}^k(a)$ does not vanish implies that by perturbing $a$ 
 slightly if necessary
we can arrange that  $c_{0j}^k(a)$ itself is nonzero.  It follows that
$pt*_a\xi_j\ne 0$.
  \end{proof}
 
 A similar argument proves the analog of the other statements
 in Lemma~\ref{le:QH}.  Moreover, if 
  $(\Aa^a,f)$ denotes the Frobenius algebra $QH^a_{2n}(M)$, then  
 one can check as before that   $(\Aa^a,f)$ satisfies the conditions of 
 Lemma~\ref{le:frob}.  Thus we deduce:
 
 \begin{prop}\label{prop:a}  Assume that condition (*) holds.
 Then  $(M,\om)$ is  uniruled with even constraints iff
there is $a\in \Hh$ such that the even quantum ring
$QH_{ev}^a(M)$ has a unit in $\Qq_-$.
 \end{prop}
 
 \begin{rmk} \rm (i)
 Suppose that $(\Aa,f)$ is a
  Frobenius algebra over a field $\FF$ with the property that $f(\1)=0$.  Suppose further that there is $p\in \Aa$ such that $f(p)\ne 0$ 
  while $f(p^2)=0$. Since the functional
  $\ker f\to \FF$ given by $x\mapsto f(px)$ does not vanish when $x=\1$, its kernel $\Mm$ is a complement to   $\1\FF $ in $\ker f$.
  Moreover the assumption $f(p^2) = 0$ implies that 
  the subspace of $\Aa$ orthogonal to $p$ 
  is $p\FF\oplus \Mm=:\Qq_-$. 
  Hence $\Aa$ decomposes additively as $p\FF\oplus \Mm\oplus \1\FF$ 
  as  in Lemma~\ref{le:frob}. 
  Therefore the conditions in this lemma are satisfied for some  $\Mm$ provided only that $f(\1)=0$ and there is $p$ with $f(p)\ne 0, f(p^2) = 0$.\SSS
  
  \NI (ii)  The fact that there is such a nice characterization of the uniruled property in terms of the structure of quantum homology leads immediately
  to speculations about rational connectedness.   A projective manifold is said to be   {\bf rationally connected} if there is a holomorphic $\PP^1$ through every generic pair of points: see Kollar~\cite{K}.  This implies that  
  there is a holomorphic $\PP^1$ through generic sets of $k$ points, for any $k$, but it is not yet known whether these spheres are visible in quantum homology, e.g. it is not known whether 
  there must be a nontrivial Gromov--Witten invariant with more than one point constraint.  This would correspond to the point class $p: = pt$ in 
  $QH_{ev}^a(M)$ having nonzero square $p^2$.  This raises many questions.  Are there symplectic manifolds with $p$ nilpotent but with
  $p^2\ne 0$?  If $p$ is not nilpotent is the 
  quantum homology semi-simple?   
  Abrams' condition for semi-simplicity in~\cite{Ab2} involves the quantum Euler class.  What is its relation to the class $p$? There are many possible choices for the coefficient ring $\La$; 
 how do these affect the situation?  \SSS
  
  \NI (iii) The purpose of Proposition~\ref{prop:a} is to show that there is not much conceptual difference between the usual quantum product and its deformations $*_a$.  All the usual applications of quantum homology
  (such as the Seidel representation and spectral invariants) should have analogs for $*_a$.  For example, given $a\in H_{2d}(M;\Q)$ let $\Gg^a$ be the extension of 
  $\pi_1(\Ham(M))$ whose elements are pairs $(\ga,\Ta)$ consisting of an element $\ga\in \pi_1(\Ham(M))$ with a class $\Ta\in H_{2d+2}(P_\ga) $ 
  such that $\Ta\cap [M] = a$; cf. the discussion of the group ${\widehat G}$ in \cite[Ch~12.5]{MS}.  Then (assuming that the appropriate version of condition (*) holds) one can define a homomorphism
  $$
  \Ss^a:\Gg^a\to \bigl(QH_{ev}^a(M)\bigr)^\times
  $$
  by setting
  $$
  \Ss^a\bigl((\ga,\Ta)\bigr) = \sum_{\si,m,i} \frac 1{m!}\bla \xi_i,\Ta,\dots,\Ta\bra_{m+1,\si}\;\;
  \xi_i^*\otimes
t^{-u_\ga(\si)},
  $$
as in equation (\ref{eq:S}). It is not hard to check that this satisfies the analog of (\ref{eq:Sa1}), namely
  $$
   \Ss^a\bigl((\ga,\Ta)\bigr)*_ab = \sum_{\si,m,i}
   \frac 1{m!}\bla b,\xi_i,\Ta,\dots,\Ta\bra_{m+2,\si}  \;\; \xi_i^*\otimes t^{-u_\ga(\si)}.
$$
  \end{rmk}

  \end{appendix}

\end{document}